\documentclass[a4,reqno,10pt]{amsart}

\usepackage{amsmath,amsthm,amssymb,xcolor,graphicx,soul}
\usepackage{amscd,verbatim,a4wide,amsfonts,bbm}

\usepackage{tikz}
\usetikzlibrary{arrows}
\usetikzlibrary{positioning}

\usepackage[hidelinks]{hyperref}
\usepackage[all]{xy}

\theoremstyle{plain}
\newtheorem{theorem}{Theorem}
\newtheorem{proposition}{Proposition}

\theoremstyle{definition}
\newtheorem{definition}[theorem]{Definition}

\theoremstyle{remark}
\newtheorem*{remark}{Remark}
\newtheorem{example}{Example}

\numberwithin{equation}{section}
\numberwithin{theorem}{section}
\numberwithin{proposition}{section}
\numberwithin{figure}{section}
\numberwithin{example}{section}

\newcommand{\cA}{{\mathcal{A}}}

\newcommand{\cE}{{\mathcal{E}}}
\newcommand{\cF}{{\mathcal{F}}}
\newcommand{\cG}{{\mathcal{G}}}
\newcommand{\cI}{{\mathcal{I}}}
\newcommand{\cJ}{{\mathcal{J}}}

\newcommand{\cL}{\mathcal{L}}
\newcommand{\cN}{\mathcal{N}}
\newcommand{\cU}{\mathcal{U}}
\newcommand{\cX}{\mathcal{X}}

\newcommand{\CC}{{\mathbb{C}}}
\newcommand{\DD}{{\mathbb{D}}}
\newcommand{\RR}{{\mathbb{R}}}
\newcommand{\TT}{{\mathbb{T}}}
\newcommand{\XX}{{\mathbb{X}}}
\newcommand{\ZZ}{{\mathbb{Z}}}

\newcommand{\sfg}{\mathsf{g}}
\newcommand{\sfG}{\mathsf{G}}
\newcommand{\sfGL}{\mathsf{GL}}
\newcommand{\sfO}{\mathsf{O}}
\newcommand{\sfS}{\mathsf{S}}
\newcommand{\sfT}{\mathsf{T}}
\newcommand{\sfU}{\mathsf{U}}

\newcommand{\fD}{{\mathfrak{D}}}

\newcommand{\fJ}{{\mathfrak{J}}}
\newcommand{\fL}{{\mathfrak{L}}}
\newcommand{\fX}{{\mathfrak{X}}}

\newcommand{\pp}{{\mathbbm p}}

\newcommand{\Ann}{{\text{Ann}}}
\newcommand{\End}{{\mathsf{End}}}

\begin{document}

\title[Generalised contact geometry]
{Generalised contact geometry as reduced generalised Complex geometry}

\author[K Wright]{Kyle Wright}
\address[K Wright]{
Department of Theoretical Physics,
Research School of Physics and Engineering, and
Mathematical Sciences Institute,
Australian National University, 
Canberra, ACT 2601, Australia}
\email{wright.kyle.j@gmail.com}

\maketitle  

\begin{abstract}
	Generalised contact structures are studied from the point of view of reduced generalised complex structures, naturally incorporating non-coorientable structures as non-trivial fibering. The infinitesimal symmetries are described in detail, with a geometric description given in terms of gerbes. As an application of the reduction procedure, generalised coK\"ahler structures are defined in a way which extends the K\"ahler/coK\"ahler correspondence. 
\end{abstract}
\section{Introduction}

Generalised geometry, introduced by Hitchin in \cite{Hit02} and developed in \cite{Gua04}, has proved to be a very successful extension of differential geometry on the generalised tangent bundle $TM\oplus T^*M$.  Generalised complex structures, defined on even
dimensional manifolds,  unify and interpolate
between symplectic and complex structures.  Much of the interest in generalised geometry, particularly in relation to T-duality in string theory, is due to the enlarged symmetry group of the structures on $TM\oplus T^*M$ and the associated deformations.  

The odd-dimensional counterpart, generalised contact geometry, is less well developed (see however \cite{Vai07,Vai08,Poo09}), particularly in the description of the associated symmetries.  This note presents generalised contact structures as $\sfS^1$-invariant reductions of generalised complex structures.  A geometric interpretation of twisted generalised contact structures is given through bundle gerbes and splitting the associated vector bundle. 

The content of the note is the following:  In Section \ref{Gengeom} generalised geometry on the generalised tangent bundle $E\cong TM\oplus T^*M$ is reviewed.  Section \ref{gencontgeom} introduces the generalised geometry associated with $E\cong TM\oplus\RR\oplus\RR\oplus T^*M$, extending the treatment in \cite{Igl01} and \cite{Gra13}, to include the full set of symmetries.  Generalised contact structures are described as $\sfS^1$-invariant reductions of generalised complex structures in Section \ref{gencontactstruc}.  The extended $\kappa$-symmetries noted by Sekiya \cite{Sek12} correspond to reductions of non-trivial $\sfS^1$-bundles.     In Section \ref{genSasakigeom} twisted generalised coK\"ahler structures are described as reductions of generalised K\"ahler structures.  This gives an generalised analogue of the correspondence between coK\"ahler structures on $M$ and K\"ahler structures on a principal circle bundle $\sfS^1\hookrightarrow P\rightarrow M$.  The role of the extended symmetries are discussed in the context of T-duality in Section \ref{Tduality}. The main result being that generalised coK\"ahler structures are mapped to other generalised coK\"ahler structures under T-duality.  Finally, in Section \ref{Linebundle}, the relationship between the twisted contact structures in this note, and geometry on the generalised derivation bundle $\DD L\cong \fD L\oplus \fJ^1L$ (introduced in \cite{Vit15b}) is given.

\section{Generalised tangent spaces and Courant algebroids}\label{Gengeom}

Generalised geometry is the study of geometric structures on a vector bundle equipped with an algebroid structure. Courant algebroids underly the generalised geometries associated with generalised complex structures and generalised contact structures.
\begin{definition}
	A \emph{Courant algebroid} is a quadruple $(E,\circ,\langle\cdot,\cdot\rangle,\rho)$, where $E\rightarrow M$ is a vector bundle, $\circ:\Gamma(E)\times \Gamma(E)\rightarrow \Gamma(E)$ is a Dorfman product, $\langle\cdot,\cdot\rangle$ a non-degenerate metric, and $\rho:E\rightarrow TM$ is an anchor, satisfying:
	\begin{align}
	e_1\circ(e_2\circ e_3)=&(e_1\circ e_2)\circ e_3+e_2\circ (e_1\circ e_3)\label{Leibniziden};\\
	\rho(e_3)\langle e_1,e_2\rangle=&\langle e_3\circ e_1,e_2\rangle+\langle e_1,e_3\circ e_2\rangle;\\
	e_1\circ e_1=&D\langle e_1, e_1\rangle;
	\end{align}
	for $e_i\in\Gamma(E)$, where $D$ is defined by $\langle Df,e_1\rangle=\frac{1}{2}\rho(e_1)f$.
\end{definition}

Standard generalised geometry is the study of geometric structures on the generalised tangent bundle $E\rightarrow M$ given by the following exact sequence:
\begin{align}\label{Gentangentsequence}
\xymatrix{0\ar[r]& T^*M \ar[r]^{\rho^*}  & E \ar[r]^{\rho} \ar@/^1pc/[l]^{s^*} & TM \ar@/^1pc/[l]^s \ar[r] & 0}.
\end{align}
Courant algebroids on $E$, specified by \eqref{Gentangentsequence}, are called \emph{exact Courant algebroids}.  Every exact Courant algebroid admits a splitting $s:TM\rightarrow E$, satisfying $\rho s=\text{id}$, and is isotropic $\langle s(X_1),s(X_2)\rangle=0$ for all $X_1,X_2\in\Gamma(TM)$.  Two exact Courant algebroids are equivalent if they differ by a choice of isotropic splitting.  

A choice of splitting defines an isomorphism $E\cong s(TM)\oplus \rho^*(T^*M):=\TT M$.  Using the identification $\Gamma(e)=s(X)+\rho^*(\xi):=(X,\xi)$, for $X\in\Gamma(TM)$ and $\xi\in \Gamma(T^*M)$, the standard Courant algebroid is given by
\begin{subequations}\label{StandardCourant}
	\begin{align}
	(X_1,\xi_1)\circ_H(X_2,\xi_2)=&\Big([X_1,X_2],\cL_{X_1}\xi_2-\iota_{X_2}\xi_1-\iota_{X_1}\iota_{X_2}H\Big);\label{Dorf}\\
	\langle(X_1,\xi_1),(X_2,\xi_2)\rangle=&\frac{1}{2}(\iota_{X_1}\xi_2+\iota_{X_2}\xi_1);\\
	\rho(X,\xi)=&X,
	\end{align}
\end{subequations}
where $H\in\Omega^3_{\text{cl}}(M)$, is given by
\begin{align*}
H(X_1,X_2)=s^*(s(X_1)\circ s(X_2)),
\end{align*}
(for details see \cite{Bou10}).  The Leibniz identity for $\circ_H$ \eqref{Dorf}, gives the Maurer-Cartan identity $dH=0$.  The Dorfman product, $\circ_H$, is natural in the sense that it is the derived bracket of $d_H:=d+H\wedge$ (with $d$ the de Rham differential), acting on $\Gamma(\wedge^\bullet T^* M)$ \cite{Kos04}. 

Consider splitting \eqref{Gentangentsequence} with two different isotropic splittings $s_i:TM\rightarrow E$, $i=1,2$, satisfying $\rho(s_1-s_2)=0$.  Exactness implies that there exists a unique $B\in\Omega^2(M)$ satisfying $s_1(X)-s_2(X)=\rho^*(B(X))$, for all $X\in\Gamma(TM)$.  It can be shown that
\begin{align*}
H_1-H_2=dB.
\end{align*}
Equivalent exact Courant algebroids are classified by $[H]\in H^3(M,\RR)$, a point first noted by \v{S}evera \cite{Sev98}.

The equivalence of the exact Courant algebroid under isotropic splittings corresponding to some $B\in \Omega^2(M)$ is closely related to the concepts of symmetries in generalised geometry structures.
 
\subsection{Courant algebroid symmetries}\label{symmetries}~\\
Perhaps the most interesting aspect of generalised geometry is the enhanced symmetry group.  The symmetry group of the Lie algebroid, given by the commutator of vector fields on $TM$, is $\mathsf{Diff}(M)$.  Exact Courant algebroids have a symmetry group given by $\mathsf{Diff}(M)\ltimes \Omega^2_{\text{cl}}(M)$ if $H=0$.
\begin{definition}
	A \emph{Courant algebroid symmetry} is a bundle automorphism $S:E\rightarrow E$ such that 
	\begin{equation}
	\langle S(\XX_1),S(\XX_2)\rangle=\langle \XX_1,\XX_2\rangle,\quad S(\XX_1)\circ S(\XX_2)=S(\XX_1\circ \XX_2).
	\end{equation}
\end{definition}
Given a diffeomorphism $f:M\rightarrow M$, the induced action on a section $(X,\xi)\in\Gamma(\TT M)$ is given by
\begin{align*}
X_p+\xi_p\rightarrow (T_pf)(X_p)+(T_{f(p)}f^{-1})^*(\xi_p).
\end{align*}
There is an infinitesimal action of $B\in\Omega^2(M)$ on $(X,\xi)\in \Gamma(\TT M)$, given by $B(X,\xi)=(0,\iota_X B)$.  The corresponding finite action, called a \emph{$B$-transformation}, is $e^B(X,\xi)=(X,\xi+\iota_XB)$, and satisfies 
\begin{align*}
\langle e^B(X_1,\xi_1),e^B(X_2,\xi_2)\rangle=&\langle (X_1,\xi_1),(X_2,\xi_2)\rangle;\\  e^B(X_1,\xi_1)\circ_He^B(X_2,\xi_2)=&e^B((X_1,\xi_1)\circ_{H+dB}(X_2,\xi_2)).
\end{align*}
If $H=0$ a $B$-transformation is a Courant algebroid symmetry iff $B\in\Omega^2_{\text{cl}}(M)$.  The Lie group composition of closed two forms is
\begin{align*}
e^{B'}e^{B''}=e^{B'+B''},\quad B'\cdot B''=B'+B''.
\end{align*}

The $H$-twisted exact Courant algebroid and the $B$-transformations have a close connection to $\sfU(1)$-gerbes.  If one requires that $H$ has integral periods, $[H/2\pi]\in H^3(M,\ZZ)$, then there is a gerbe describing the patching of $T^*M$ to $TM$, giving $\TT M$ \cite{Hit02}.

Consider a good cover of $M$, denoted $\cU=\cup_\alpha U_\alpha$, with $U_{\alpha\beta\dots\gamma}=U_\alpha\cap U_\beta\cdots\cap U_\gamma$.  A  gerbe is described by the cocycle $g_{\alpha\beta\gamma}=\exp(i\Lambda_{\alpha\beta\gamma})\in \sfU(1)$, a connection is given by $A_{\alpha\beta}\in\Omega^1(U_{\alpha\beta})$ and $B_\alpha\in\Omega^2(U_\alpha)$, satisfying
\begin{align*}
B_\beta-B_\alpha=&dA_{\alpha\beta}\quad \ {\text{on }}U_{\alpha\beta},\\
A_{\beta\gamma}-A_{\alpha\gamma}+A_{\alpha\beta}=&d\Lambda_{\alpha\beta\gamma}\quad {\text{on }}U_{\alpha\beta\gamma}.
\end{align*}
$H=dB_\alpha=dB_{\beta}$ on $U_{\alpha\beta}$ is independent of the cover and is a globally defined 3-form.

Given a representative of a class $[H]\in H^3(M,\RR)$ it is possible to reconstruct the bundle $E$.  Choose an open cover $\cU$ and a representative $H\in H^3_{\text{dR}}(M,\RR)$, which consists of a 4-tuple $(\Lambda_{\alpha\beta\gamma},A_{\alpha\beta},B_\alpha,H)$ in the \v{C}ech-de Rham complex (over $\RR$).  The bundle $E$ is constructed by the clutching construction
\begin{align*}
E=\bigsqcup_\alpha(TU_\alpha\oplus T^*U_\alpha)/\sim,
\end{align*}
identifying $(X_\alpha,\xi_\alpha)\in (\Gamma(TU_\alpha), \Gamma(T^*U_\alpha))$ with $(X_\beta,\xi_\beta)\in (\Gamma(TU_\alpha), \Gamma(T^*U_\alpha))$ on overlaps iff $X_\beta=X_\alpha$ and $\xi_\beta=\xi_\alpha+\iota_{X_\alpha}dA_{\alpha\beta}$. Consistency on triple overlaps $U_{\alpha\beta\gamma}$ follows from $dA_{\beta\gamma}-dA_{\alpha\gamma}+dA_{\alpha\beta}=(d\delta A)_{\alpha\beta\gamma}=(d^2\Lambda)_{\alpha\beta\gamma}=0$.  On $TM|_{U_\alpha}$ the splitting is given by 
\begin{align*}
s|_{U_\alpha}:X_\alpha\rightarrow X_\alpha+\iota_{X_\alpha}B_\alpha,
\end{align*}
and consistency on overlaps follows from $B_\beta-B_\alpha=dA_{\alpha\beta}$.

A closed $B$-field such that $B/2\pi$ has integral periods should be considered as a gauge transformation.  This means that the notion of equivalence of generalised structures should not be just the diffeomorphisms connected to the identity, but extended by $H^2(M,\ZZ)$. 

The importance of the $H$-twist appearing in the exact Courant algebroid \eqref{Dorf} is in identifying it as $H(X_1,X_2)=s^*(s(X_1),s(X_2))$ for some non-trivial bundle $E\rightarrow M$.

\begin{remark}
	In string theory applications the requirement that $[H/2\pi]\in H^3(M,\ZZ)$ arises naturally as the requirement ensuring a single valued path-integral.
	$[H/2\pi]\in H^3(M,\ZZ)$ is interpreted as the Neveu-Schwarz flux, with a local two-form potential, $B$, satisfying $H|_{U_\alpha}=dB_\alpha$. $B$-transformations in $H^2(M,\ZZ)$ are viewed as gauge transformations, and should be quotiented out when considering physically distinct states.  Generalised geometry can be seen as a way of encoding flux geometrically.
\end{remark}
\subsection{Generalised geometric structures}\label{gengeomstructures}~\\
Almost all differential geometry structures have a counterpart in generalised geometry, defined on the vector bundle $\TT M$.  Of most interest are generalised complex structures, generalised metric structures, and generalised K\"ahler structures, with \cite{Gua04} being the standard introductory reference. 

In addition to the exact Courant algebroid structure, the generalised tangent bundle admits a Clifford action of sections $(X,\xi)\in \Gamma(\TT M)$ on differential forms $\varphi\in\Omega^\bullet(M)$, given by
\begin{equation}
(X,\xi)\cdot\varphi=\iota_X\varphi+\xi\wedge\varphi,
\end{equation}
satisfying $(X,\xi)^2\cdot\varphi=\langle(X,\xi),(X,\xi)\rangle \varphi:=||(X,\xi)||^2\varphi$.  Forms $\varphi\in (\wedge^m T^*)^{\frac{1}{2}}\otimes\Omega^{\text{od/ev}}(M)$, where $m=\dim(M)$, describe spinors \cite{Gua04}.  Associated to each spinor is the annihilator bundle
\begin{align*}
L_\varphi:=\Ann(\varphi)=\{(X,\xi)\in \Gamma(\TT M): (X,\xi)\cdot\varphi=0 \}.
\end{align*}  
A \emph{pure spinor} is given by $\varphi\in \Omega^{\text{od/ev}}(M)\otimes \CC$ that is non-degenerate with respect to the Mukai pairing $(\varphi,\varphi)_M\neq 0$, where
\begin{equation}
(\varphi_1,\varphi_2)_M=(\alpha(\varphi_1)\wedge \varphi_2)_{m},
\end{equation}
where $\alpha$ is the Clifford anti-automorphism $\alpha(dx^1\otimes dx^2\otimes\cdots\otimes dx^k)=dx^k\otimes dx^{k-1}\otimes\cdots\otimes dx^1$, $m=\dim(M)$, and $(\cdot)_m$ denotes the projection onto $\Omega^m(M)$.  Given a pure spinor $\varphi$, and a function $f\in C^\infty(M)$ we have
\begin{align*}
(X,\xi)\cdot f\varphi=f(X,\xi)\cdot \varphi,\quad (f\varphi,f\varphi)_M=f^2(\varphi,\varphi)_M.
\end{align*}
If $f$ is nowhere zero, then $f\varphi$ describes the same maximal isotropic subspace as $\varphi$.  There is a local one-to-one correspondence between maximally isotropic subspaces of $\TT M$ and conformal classes of pure spinors.  

\begin{definition}
	A \emph{generalised almost complex structure} on $\TT M$ is given by $\cJ\in \End(\TT M)$, satisfying $\cJ^*=-\cJ$ and $\cJ^2=-\text{id}$. 
\end{definition}
A generalised almost complex structure can equivalently be described by a maximal isotropic complex subbundle $L\subset \TT M\otimes \CC$, satisfying $L\cap \bar{L}=\{ 0\}$, for 
\begin{align*}
L_\cJ=\{\XX\in \Gamma(\TT M\otimes\CC):\cJ(\XX)=i\XX \}.
\end{align*}
There is a local one-to-one correspondence between generalised almost complex structures and conformal classes of complex pure spinors, where a complex pure spinor satisfies the non-degeneracy condition $(\bar{\varphi},\varphi)_M\neq 0$.

A generalised almost complex structure, $\cJ$, is $H$-involutive if all sections $\XX$ of the $+i$-eigenbundle $L_J$ are involutive with respect to $\circ_H$: $\XX_1,\XX_2\in L_\cJ \Rightarrow \XX_1\circ_H \XX_2\in L_\cJ$.

\begin{definition}
	A \emph{generalised complex structure} is an $H$-involutive generalised almost complex structure. 
\end{definition}

A pure spinor, $\varphi$, is said to be \emph{$H$-involutive} if there exists some $v\in \Gamma(\TT M\otimes \CC)$ such that
\begin{equation}
d\varphi+H\wedge \varphi=v\cdot\varphi.
\end{equation} 
Given an $H$-involutive pure spinor $\varphi$ and a nowhere zero complex function $f\in C^\infty(M)$, we have
\begin{align*}
df\varphi+H\wedge f\varphi=f(d\varphi+H\wedge\varphi)+df\wedge \varphi=fv\cdot \varphi+df\wedge \varphi:=v'\cdot f\varphi,
\end{align*}
where $v'=v+f^{-1}df$.

There is a local one-to-one correspondence between Generalised complex structures and conformal classes of $H$-involutive pure spinors.

\begin{example}
	Given an (almost) symplectic structure $\omega\in\Omega^2(M)$, we can define a generalised (almost) complex structure with the spinor,	$\varphi_\omega=e^{i\omega}$.  $(\varphi_\omega,\bar{\varphi}_\omega)_M=\omega^{m/2}\neq 0$ as $\omega$ is non-degenerate.  The $+i$-eigenbundle is given by  
	\begin{align*}
	L_\omega=\{(X,\xi)\in \Gamma(\TT M\otimes \CC) :\xi=i\omega(X,\cdot)\}.
	\end{align*}
\end{example}
\begin{example}
	Given an (almost) complex structure $J\in\End(TM)$, satisfying, $J^*=-J$, $J^2=-\text{id}$, we can define a generalised (almost) complex structure with the spinor $\varphi_J=\Omega$,
	where $\Omega\in \Omega^{(m,0)}(M)$, is any generator of the $(m,0)$-forms for the complex structure $J$.  $(\rho_J,\bar{\rho}_J)_M=\Omega\wedge \bar{\Omega}$.  The $+i$-eigenbundle is given by  
	\begin{align*}
	L_\Omega=T^{(0,1)}M\oplus T^*{}^{(1,0)}M
	\end{align*}
\end{example}

The tangent bundle $TM$ has a structure group $\sfGL(m)$. A reduction of the structure group $\sfGL(m)$ to its maximal compact subgroup $\sfO(m)$, defines a choice of Riemannian metric.  The generalised tangent bundle $\TT M$, equipped with metric $\langle\cdot,\cdot\rangle$, has structure group $\sfO(m,m)$.   
\begin{definition}
	A \emph{Generalised metric} $\cG$ is a positive definite metric on $\TT M$, corresponding to a choice of reduction of the structure group from $\sfO(m,m)$ to $\sfO(m)\times \sfO(m)$.  
\end{definition}
The inner product $\langle\cdot,\cdot\rangle$ determines a splitting $\TT M=C_+\oplus C_-$, where $C_+$ is positive definite with respect to $\langle\cdot,\cdot\rangle$, and $C_-$ is negative definite. The generalised metric structure is defined by
\begin{align}\label{GenMetric}
\cG(\XX_1,\XX_2):=\langle\XX_1,\XX_2\rangle|_{C_+}-\langle\XX_1,\XX_2\rangle|_{C_-}.
\end{align}

Using the metric $\langle\cdot,\cdot\rangle$ to identify $\TT M$ with $\TT^*M$, a generalised metric can be identified with $\cG\in \End(\TT M)$ satisfying $\cG^*=\cG$, and $\cG^2=\text{id}$. It follows from \eqref{GenMetric} that $C_\pm$ correspond to the $\pm 1$-eigenbundles of $\cG$.

Given a Riemannian metric $\sfg$, a generalised metric $\cG$ can be defined by the identification 
\begin{align*}
C_\pm=\{(X,\xi)\in \TT M: \xi=\pm \sfg(X,\cdot) \}.
\end{align*}  
  
\begin{definition}
	A \emph{generalised almost K\"ahler structure} is a pair of almost generalised structures satisfying $\cJ_1\cJ_2=\cJ_2\cJ_1$, and $-\cJ_1\cJ_2=\cG$, for a generalised metric $\cG$. 
\end{definition}

The role of extended symmetry is important as it provides deformations of generalised metric and generalised complex structures.  This provides a way to generate examples, and gives a notion of equivalence which goes beyond diffeomorphisms.

\begin{example}[Twisted generalised metric]
	A generalised metric, defined by a Riemannian metric, $\sfg$, can be twisted by a 2-form $B$, to give another generalised metric.
	\begin{align*}
	C^B_\pm=\{(X,\xi) \in \TT M: \xi=\pm \sfg(X,\cdot)+B(X,\cdot) \}.
	\end{align*}
\end{example}
\begin{example}[Twisted generalised complex structure]\label{localgencomplex}
	Given a pure spinor $\varphi$ defining an (almost) complex structure, we can define $\varphi^B=e^{-B}\wedge\varphi$, for $B\in\Omega^2(M)$, satisfying
	\begin{align*}
	(\varphi^B,\bar{\varphi}^B)_M=(\varphi,\bar{\varphi})_M,\quad d_{H'}(e^{-B}\wedge\varphi)-v'\cdot(e^{-B}\wedge \varphi)=e^{-B}(d_H\varphi-v\cdot \varphi),
	\end{align*}
	where $H'=H+dB$, and $v'=e^{B}\cdot v=e^B(X,\xi)=(X,\xi+\iota_X B)$. 
\end{example}
A generalised $H$-involutive complex structure $\varphi$ can be deformed to a $(H+dB)$-involutive complex structure $e^{-B}\wedge\varphi$.

A generalised (almost) complex structure is of \emph{geometric type-$k$}, if $t_L(x):=\text{codim}_\CC(\rho(L_x))$.  Generically the type can change at each point $x\in M$.
\begin{example}
	Locally every generalised (almost) complex structure of type-$k$, can be associated (non-canonically) to a pure spinor $\varphi_J=\Omega\wedge e^{B+i\omega}$, where $\omega\in\Omega^2(M)$, $\Omega$ is a complex decomposable form of degree $k$, and $\omega^{m/2-k}\wedge \Omega\wedge\bar{\Omega}\neq 0$.
\end{example}
  
In Section \ref{genSasakigeom} generalised coK\"ahler structures will be defined in a way that mirrors the common definition in terms of K\"ahler structures.  This Section concludes with the definition of generalised Calabi-Yau structures, and hyperK\"ahler structures. 

\begin{definition}
	A \emph{generalised almost Calabi-Yau structure} consists of two pure spinors $(\varphi_1,\varphi_2)$, which define two generalised complex structures $(\cJ_1,\cJ_2)$ forming a generalised almost K\"ahler structure. In addition, the lengths of these sections are related by a constant 
	\begin{align*}
	(\varphi_1,\bar{\varphi}_1)_m=c(\varphi_2,\bar{\varphi}_2)_m,
	\end{align*}
	where $c\in\RR$ can be scaled to either $+1$ or $-1$ by rescaling $\varphi_1$. 
\end{definition}

A generalised Calabi-Yau structure is a generalised almost Calabi-Yau structure where $(\varphi_1,\varphi_2)$ are both $H$-involutive.

\begin{example}[Calabi-Yau]\label{CalabiYauEx}
	A Calabi-Yau manifold is a K\"ahler manifold of complex dimension $m$ with symplectic form $\omega$ and holomorphic volume form $\Omega$ satisfying $\omega^m=2^{-m}i^mm!\Omega\wedge\bar{\Omega}$.  This gives a generalised Calabi-Yau structure with $\varphi_1=e^{i\omega}$, and $\varphi_2=\Omega$, satisfying
	\begin{align*} (e^{i\omega},e^{-i\omega})=(-1)^{\frac{m(m-1)}{2}}(\Omega,\bar{\Omega}).
	\end{align*}
\end{example}

\begin{example}[hyperK\"ahler]
	Given a hyperK\"ahler structure $(M,g,I,J,K)$ a generalised K\"ahler structure can be constructed:
	\begin{align*}
	\varphi_1=e^{B+i\omega_1},\quad \varphi_2=e^{-B+i\omega_2},
	\end{align*}
	where $B=\omega_K,\ \omega_1=\omega_I-\omega_J,\ \omega_2=\omega_I+\omega_J$.
\end{example}

\section{Generalised contact geometry}\label{gencontgeom}
The $H$-twisted exact Courant algebroid, and the associated Dirac structures play a fundamental role in generalised complex geometry.  The corresponding objects in generalised contact geometry are the contact Courant algebroid, and contact Dirac structures.  

For sections $\XX=(X,f,g,\xi)\in (\Gamma(TM),C^\infty(M),C^\infty(M),\Gamma(T^*M))$, the \emph{contact Courant algebroid} is given by \cite{Bou05,Bou10}:
\begin{subequations}\label{RedCouranttot} 
	\begin{align}
	\XX_1\circ_{H_3,H_2,F}\XX_2=\Big(&[X_1,X_2],X_1(f_2)-X_2(f_1)-\iota_{X_1}\iota_{X_2}F,X_1(g_2)-X_2(g_1)-\iota_{X_1}\iota_{X_2}H_2,\nonumber\\
	&\cL_{X_1}\xi_2-\iota_{X_2}d\xi_1-\iota_{X_1}\iota_{X_2}H_3+g_2df_1+f_2dg_1\label{RedCourant}\\
	&+f_1\iota_{X_2}H_2-f_2\iota_{X_1}H_2+g_1\iota_{X_2}F-g_2\iota_{X_1}F\Big);\nonumber\\
	\langle \XX_1,\XX_2\rangle=&\frac{1}{2}(\iota_{X_1}\xi_2+\iota_{X_2}\xi_1+f_1g_2+g_1f_2);\label{innerprod}\\
	\rho(\XX)=&X;
	\end{align}
\end{subequations}
where the twists $(H_3,H_2,F)\in(\Omega^3(M),\Omega^2(M),\Omega^2(M))$ are globally defined forms required to satisfy the Maurer-Cartan identities:
\begin{align}\label{MaurerCartan} 
dH_3+H_2\wedge F=0,\quad dH_2=0,\ dF=0.
\end{align}
This is a twisted version of the contact Courant algebroid that has appeared previously in the generalised contact literature \cite{Igl01,Igl04,Gra13}.  The twists $(H_3,H_2,F)$ play an essential role in describing symmetries and deformations of generalised contact structures.

First consider the case that $H_3=H_2=F=0$.  There is an action of $B\in\Omega^2(M)$, $a,b\in\Omega^1(M)$ on $\XX=(X,f,g,\xi)\in\Gamma(E)$:
\begin{align}
e^{(B,b,a)}(X,f,g,\xi)=&\Big(X,f+2\langle X,a\rangle,g+2\langle X,b\rangle,\xi+\iota_X B-fb-ga-\langle X,a\rangle b-\langle X,b\rangle a\Big).\label{Symm}
\end{align}
The action satisfies
\begin{equation*}
\langle e^{(B,b,a)}\XX_1,e^{(B,b,a)}\XX_2\rangle=\langle\XX_1,\XX_2\rangle,\quad e^{(B,b,a)}\XX_1\circ_{(0,0,0)} e^{(B,b,a)}\XX_2=e^{(B,b,a)}(\XX_1\circ_{H'_3,H'_2,F'}\XX_2)\label{RedCourantsymm},
\end{equation*}
where 
\begin{equation*}
H'_3=dB+\frac{1}{2}(da\wedge b+a\wedge db),\quad H'_2=db, \quad F'=da.
\end{equation*}
When $H_3=H_2=F=0$ the  bracket \eqref{RedCourant} has the symmetry group
\begin{align*} \mathsf{Diff}(M)\ltimes (\Omega^2_{\text{cl}}(M),\Omega^1_{\text{cl}}(M),\Omega^1_{\text{cl}}(M)).
\end{align*}
The group action is generated by the algebra action
\begin{align*}
(B,b,a)\cdot(X,f,g,\xi)=(0,\iota_{X}a,\iota_{X}b,\iota_{X}B-fb-ga),
\end{align*}
with 
\begin{align*}
e^{(B,b,a)}\XX_1=\XX_1+(B,b,a)\cdot\XX_1+\frac{1}{2!}(B,b,a)^2\cdot\XX_1+\dots
\end{align*}
where the algebra composition is given by
\begin{align*}
(B_2,b_2,a_2)\cdot(B_1,b_1,a_1)=(B_1+B_2-\frac{1}{2}(b_1\wedge a_2+a_1\wedge b_2),b_1+b_2,a_1+a_2),
\end{align*}
and is non-abelian.  

For non-trivial $(H_3,H_2,F)$ the symmetries are described by a non-abelian gerbe.  The construction follows from Baraglia's general argument for twisting closed form Leibniz algebroids \cite{Bar11}.  The twists are constructed from
\begin{align}\label{CurvBba}
(H_3,H_2,F)=(dB_\alpha-\frac{1}{2}a_\alpha\wedge db_\alpha-\frac{1}{2}da_\alpha\wedge b_\alpha ,db_\alpha,da_\alpha),
\end{align}
where 
\begin{align*}
(B_\alpha,b_\alpha,a_\alpha)\in(\Omega^2(U_\alpha),\Omega^1(U_\alpha),\Omega^1(U_\alpha)),
\end{align*}
are required to satisfy \eqref{MaurerCartan} and satisfy cocycle conditions 
\begin{align}\label{CocycleBba}
(B_{\alpha\beta},b_{\alpha\beta},a_{\alpha\beta})\cdot(B_{\beta\gamma},b_{\beta\gamma},a_{\beta\gamma})\cdot(B_{\gamma\alpha},b_{\gamma\alpha},a_{\gamma\alpha})=0\quad\text{on }U_{\alpha\beta\gamma},
\end{align}
where 
\begin{align*}
(B_{\alpha\beta},b_{\alpha\beta},a_{\alpha\beta})=(B_{\alpha}-B_{\beta}+\frac{1}{2}b_\alpha\wedge a_\beta+\frac{1}{2}a_\alpha\wedge b_\beta,b_{\alpha}-b_{\beta},a_\alpha-a_{\beta}).
\end{align*}
This defines a twisted bundle with sections patched using  
\begin{align*}
(X_\alpha,f_\alpha,g_\alpha,\xi_\alpha)=e^{(B_{\alpha\beta},b_{\alpha\beta},a_{\alpha\beta})}(X_\beta,f_\beta,g_\beta,\xi_\beta)\quad \text{on }U_{\alpha\beta},
\end{align*}
to define a global section $(X,f,g,\xi)\in\Gamma(E)$.  The twists do not define $(B,b,a)$ uniquely, and it is possible to make a different choice $(B',b',a')$ as long as $(B_{\alpha\beta},b_{\alpha\beta},a_{\alpha\beta})=(B'_{\alpha\beta},b'_{\alpha\beta},a'_{\alpha\beta})$.  This gives the relation
\begin{align}
(B'_\alpha,b'_\alpha,a'_\alpha)=(B_\alpha+B''-\frac{1}{2}b_\alpha\wedge a''-\frac{1}{2}a_\alpha\wedge b'',b_\alpha+b'',a_\alpha+a''),
\end{align} 
where $(B'',b'',a'')$ are globally defined forms satisfying
\begin{align}\label{GaugeBba}
(dB''+H_3-\frac{1}{2}da''\wedge b''-\frac{1}{2}a''\wedge db'',db''+H_2,da''+F_2)=0.
\end{align}
\begin{definition}
	For a fixed set of twists $(H_3,H_2,F)\in (\Omega^3(M),\Omega^2(M),\Omega^2(M))$ satisfying \eqref{MaurerCartan}, a \emph{$(B,b,a)$-transformation} corresponds to a choice of $(B_\alpha,b_\alpha,a_\alpha)\in(\Omega^2(U_\alpha),\Omega^1(U_\alpha),\Omega^1(U_\alpha))$ which generate the twists $(H_3,H_2,F)$ i.e., satisfying conditions \eqref{CurvBba} and \eqref{CocycleBba}.  The choice of $(B,b,a)$-transformation is not unique.  Two transformations $(B,b,a)$ and $(B',b',a')$ will produce the same twists $(H_3,H_2,F)$ if they are related by a set $(B'',b'',a'')\in (\Omega^2(M),\Omega^1(M),\Omega^1(M))$ satisfying \eqref{GaugeBba}. A transformation $(B'',b'',a'')$ satisfying \eqref{GaugeBba} defines  a \emph{$(B'',b'',a'')$-gauge transformation}. 
\end{definition}

The description above shows that the twisted contact Courant algebroid can be seen as a bracket on via a clutching function construction.  The insight here is to interpret the contact Courant algebroid bracket as an $\sfS^1$-reduction of the standard twisted Courant algebroid.  The identification is made as follows: Consider a Courant algebroid associated to the vector bundle $E$ given as 
\begin{align*}
\xymatrix{0\ar[r] & T^*P \ar[r]^-{\rho^*_B} & E \ar[r]^-{\rho_B} \ar@/^1pc/[l]^-{(s_B)^*} & TP \ar@/^1pc/[l]^-{s_B} \ar[r] & 0,}
\end{align*}
with the standard $H$-twisted Courant algebroid structure \eqref{StandardCourant}, identifying
\begin{align*} 
H_3(\cX_1,\cX_2)=s^*_B(s_B(\cX_1),s_B(\cX_2)),\quad \cX_1,\cX_2\in \Gamma(TP).
\end{align*}
If $P(M,\pi,\sfU(1))$ is a principal $\sfU(1)$-bundle, then there are Atiyah algebroids associated to $TP$ and $(TP)^*\cong T^*P$:
\begin{align*}
&\xymatrix{0\ar[r] & P\times \RR \ar[r]^-{r_a} & TP/\sfU(1) \ar[r]^-{\pi_*} \ar@/^1pc/[l]^-{t_a} & TM \ar@/^1pc/[l]^{s_a} \ar[r] & 0,}\\
&\\
&\xymatrix{0\ar[r] & P^*\times \RR \ar[r]^-{(t_b)^*} & T^*P/\sfU(1) \ar[r]^{(s_b)} \ar@/^1pc/[l]^{(r_b)^*} & T^*M \ar@/^1pc/[l]^{(\pi_b)^*} \ar[r] & 0.}
\end{align*}
identifying
\begin{align*}
H_2(X_1,X_2)=t^*_b(s_b(X_1),s_b(X_2)),\quad F(X_1,X_2)=t^*_a(s_a(X_1),s_a(X_2)), 
\end{align*}
for $X_1,X_2\in\Gamma(TM)$.  The induced map on sections is
\begin{align*}
\xymatrix{0\ar[r]& C^\infty(P,\RR)^{\sfU(1)}\simeq C^\infty(M) \ar[r]^-r & \fX_\sfG(P)\simeq \fX(M)\oplus C^\infty(M)  \ar[r]^-{\pi_*}  & \fX(M)  \ar[r] & 0},
\end{align*}	
allowing the identification
\begin{align*}
\Gamma(E)\cong\Gamma(TP/\sfU(1))\oplus\Gamma((TP/\sfU(1))^*)\cong \Gamma(TM)\oplus C^\infty(M)\oplus C^\infty(M)\oplus \Gamma(T^*M).
\end{align*}
In this way the twisted contact Courant algebroid can be identified with 
\begin{align*}
\xymatrix{0\ar[r]& (TP/\sfU(1))^* \ar[r] & E  \ar[r]  & TP/\sfU(1) \ar[r] & 0}.
\end{align*}
The twisted contact Courant algebroid is constructed out of $(H_3,H_2,F)$ and makes no reference to $(B_\alpha,b_\alpha,a_\alpha)$.  Another choice $(B'_\alpha,b'_\alpha,a'_\alpha)$ giving the same $(H_3,H_2,F)$ will give an isomorphic twisted contact Courant algebroid.  Any two choices  $(B_\alpha,b_\alpha,a_\alpha)$ and  $(B'_\alpha,b'_\alpha,a'_\alpha)$ give the same twists iff they are related by a $(B'',b'',a'')$-gauge transformation.  The notion of equivalence of generalised contact geometry should be extended to include to the full set of symmetries---diffeomorphisms and $(B'',b'',a'')$-gauge transformations.  Geometrically the gauge transformations can be interpreted as a change in splitting of the bundle $E\cong TM\oplus \RR\oplus \RR \oplus T^*M$.  Geometry on the bundle $E$ should not be dependent on the choice of splitting.

\begin{remark}
	Sekiya studied generalised contact structures associated to the trivial line bundle $L=M\times \RR$, and noted $\kappa$-symmetries \cite{Sek12}.  This corresponds to $(b,a)$-transformations for globally defined forms not subject to periodicity constraints.  This clarifies the geometric origin of Sekiya's $\kappa$-symmetries.  The $(b,a)$-transformations correspond to a choice of connection for the circle/line-bundle $P\rightarrow M$.  The non-abelian composition law for $B$, with $(b,a)$, reflects the fact that there is choice in which order one splits the sequences.  
\end{remark}

\section{Generalised contact structures}\label{gencontactstruc}
This section describes the mixed pair description of generalised contact structures, the odd-dimensional analogue of the pure spinor description in generalised complex geometry.  Aldi and Grandini gave a proposal for mixed pairs which were compatible with $B$-transformations, but not the full set of $(B,b,a)$-transformations, so do not incorporate non-coorientable structures.

There is a Clifford action of sections $(X,f,g,\xi)\in \Gamma(TM,\RR,\RR,T^*M)$ on pairs of differential forms $(\varphi,\psi)\in\Omega^\bullet(M)$, given by
\begin{equation}
(X,f,g,\xi)\cdot (\varphi,\psi)=((X,\xi)\cdot\varphi+f\varphi,g\varphi-(X,\xi)\cdot \psi),
\end{equation}
where $(X,\xi)\cdot\varphi=\iota_X\varphi+\xi\wedge \varphi$, is the Clifford product on $\TT M$.  This product satisfies
\begin{align*}
(X,f,g,\xi)^2\cdot(\varphi,\psi)=(\iota_X\xi+fg)(\varphi,\psi)=\langle(X,f,g,\xi),(X,f,g,\xi)\rangle(\varphi,\psi):=||(X,f,g,\xi)||^2(\varphi,\psi).
\end{align*}

It is interesting to consider the annihilator bundles of a pair $(\varphi,\psi)$:
\begin{align*}
\Ann((\varphi,\psi)):=\{(X,f,g,\xi)\in E\otimes \CC: (X,f,g,\xi)\cdot(\varphi,\psi)=0 \}.
\end{align*} 
When $f=g=0$, $(X,0,0,\xi)\cdot(\varphi,\psi)=0$ implies that $(X,\xi)\cdot\varphi=0$ and $(X,\xi)\cdot \psi=0$, the same annihilator condition as Section \ref{Gengeom}.  For some pairs $(\varphi,\psi)$, there may be solutions for non-zero $f$ or $g$.  In this case
\begin{align*}
f\psi=-(X,\xi)\cdot\varphi,\quad g\varphi=(X,\xi)\cdot\psi,
\end{align*}
indicating that there exist sections $v\in \TT M$ which relate $\varphi$ and $\psi$.  Pure spinors play an important role in describing Dirac structures in $\TT M\otimes \CC$, mixed pairs describe the odd-dimensional analogue of Dirac structures.  

\begin{definition}
	A \emph{contact Dirac structure}, on an odd-dimensional manifold $M$ ($m=\dim(M)$), is a splitting of $E\otimes \CC$ into isotropic subspaces given by the triple $(L,v_1,v_2)$, where $\dim_\RR(L)=m-1$, and $v_1,v_2\in \Gamma(\TT M)$ satisfying
	\begin{align*}
	E\otimes \CC=L\oplus \bar{L} \oplus \CC v_1\oplus \CC v_2,\quad \quad 
	L\cap \bar{L}=0. 
	\end{align*} 
\end{definition}

The pairing $(\cdot\ ,\cdot)_M$ for two pairs of differential forms  $(\varphi_i,\psi_i)$ ($i=1,2$) is given by
\begin{equation}
((\varphi_1,\psi_1),(\varphi_2,\psi_2))_M:=(-1)^{|\varphi_1|}(\alpha(\varphi_1)\wedge\psi_2)_{m-1}+(-1)^{|\psi_1|}(\alpha(\psi_1)\wedge\varphi_2)_{m-1},
\end{equation}
where $\alpha(dx^1\otimes dx^2\otimes\cdots\otimes dx^k)=dx^k\otimes dx^{k-1}\otimes\cdots\otimes dx^1$,
$m=\dim(M)$, $|\varphi|=k$ for $\varphi\in\Omega^k(M)$, and $(\cdot\ )_{m-1}$ is the projection to $\Omega^{m-1}(M)$.

\begin{definition}
	A \emph{Dirac pair} $(\varphi,\psi)$ consists of two differential forms $\varphi,\psi\in\Omega^{\text{ev/od}}(M)\otimes\CC$, satisfying
	\begin{align*}
	(\alpha(\varphi)\wedge\bar{\varphi})_{m-1}\neq 0,\quad (\alpha(\psi)\wedge\bar{\psi})_{m-1}\neq &0,\quad ((\varphi,\psi),(\bar{\varphi},\bar{\psi}))_{M}\neq 0\\
	v_1\cdot\varphi=\psi,\quad& v_2\cdot\psi=\varphi,
	\end{align*} 
	for some $v_1,v_2\in \Gamma(\TT M)$.
\end{definition}
The second condition requires that either $\varphi\in\Omega^{\text{ev}}(M)\otimes\CC$ and $\psi\in\Omega^{\text{odd}}(M)\otimes\CC$, or $\varphi\in\Omega^{\text{odd}}(M)\otimes\CC$ and $\psi\in\Omega^{\text{ev}}(M)\otimes\CC$.

Given two nowhere zero functions $f_1,f_2\in C^\infty(M)$ and a Dirac pair $(\varphi,\psi)$, the pair $(f_1\varphi,f_2\psi)$ satisfy the non-degeneracy condition, and  
\begin{align*}
v_1\cdot f_1\varphi-f_2\psi=f_2(v'_1\cdot \varphi-\psi)=0,\quad& v_2\cdot f_2\psi-f_1\varphi=f_1(v_2\cdot \psi-\varphi)=0,
\end{align*} 
for some $v'_1=f_2/f_1v_1$ and $v_2=f_1/f_2 v_2$.  Thus $(\varphi,\psi)$ and $(f_1\varphi,f_2\psi)$ describe the same contact Diract structure $(L,v_1,v_2)$.

To motivate the definition of generalised contact structures, it is helpful to briefly consider the relationship between contact structures and symplectic structures.  A contact stucture is a maximally non-integrable codimension-1 hyperplane distribution $D\subset TM$.  This can be described by the line bundle $TM/D$.  Letting $\eta$ be an $TM/D$-valued 1-form, the distribution is given by $D=\ker(\eta)$.  The non-integrability condition can be given as $\eta\wedge (d\eta)^{m}\neq 0$, where $\dim(M)=2m+1$.  Letting $\omega_D=d\eta$, there is a transverse symplectic structure on $D$.  In addition, there is another symplectic structure associated with the manifold $N:=M\times \RR_{t}$; take $\alpha=dt+\eta$, and set $\omega_t=d(e^t\alpha)$.  When $TM/D$ is a non-trivial line bundle there is no globally defined contact form $\eta$.  It is possible to consider the same construction with $\sfS^1\hookrightarrow P'\rightarrow M$.  In this case there is an Atiyah algebroid structure and the contact structure can be associated with an $\sfS^1$-invariant reduction.  In the non-trivial case $\eta$ is no longer globally defined but is part of a $(B,b,a)$-transformation with $B=0,a=b=\eta$ with $H_3=0,H_2=F=d\eta$.  
 
The ability to construct two symplectic structures from a contact structure is the guiding principle of generalised contact structures.  A generalised contact structure should be able to be viewed as a (possibly non-trivial) $\sfS^1$-reduction of a generalised complex structure (see Examples \ref{cosymp} and \ref{contact}).  In addition, the definition should be compatible with $(B,b,a)$-transformations. 

\begin{remark}
Contact structures are usually associated to symplectic structures via line bundles $TM/D$, and $N=M\times \RR$.  Throughout this note $\sfS^1$-bundles will be considered primarily.  The motivation for this is the fact that $[H_2/2\pi],[F/2\pi]\in H^2(M;\ZZ)$ have a nice interpretation in terms of gerbes, as outlined in Section \ref{gencontgeom}.  The corresponding Courant algebroid description applicable to non-trivial line bundles was given by Vitagliano and Wade \cite{Vit15b}, and is briefly described in Section \ref{Linebundle}.  
\end{remark}

Generalised contact structures have been studied in a number of papers \cite{Igl04,Poo09,Ald13}.  However, the $(b,a)$-twists (which allow the description of non-coorientable structures when $H_2,F\neq 0$) have received little attention.  
\begin{definition}[\cite{Sek12}]
	A \emph{Sekiya quadruple} on an odd-dimensional manifold $M$ is given by the quadruple $(\Phi,e_1,e_2,\lambda)\in(\mathsf{End}(\TT M),\Gamma(\TT M),\Gamma(\TT M),C^\infty(M))$, satisfying the following conditions:
	\begin{subequations}
	\begin{align}
		\langle e_1,e_1 \rangle=&0=\langle e_2,e_2\rangle,\ \langle e_1,e_2\rangle=\frac{1}{2};\\
		\Phi^*=&-\Phi;\\
		\Phi(e_1)=&\lambda e_1,\ \Phi(e_2)=-\lambda e_2;\\
		\Phi^2(v)=&-v+2(1+\lambda^2)(\langle v,e_2\rangle e_1+\langle v,e_1\rangle e_2),\text{ for }v\in \Gamma(\TT M). 
	\end{align}
	\end{subequations}
\end{definition}
Generalised contact structures coming from Sekiya quadruples with $\lambda=0$ have been well studied, and are often referred to as Poon-Wade triples \cite{Poo09}.  The importance of considering $\lambda\neq 0$ is the inclusion of the $(b,a)$-symmetries, which should be considered on an equal footing to $B$-transformations, and a fundamental part of the theory.

\begin{definition}
	Let $M$ be an odd-dimensional manifold of dimension $m$.  A \emph{generalised almost contact structure} is a quadruple $(L,e_1,e_2,\lambda)$, where $L\subset \TT M\otimes \CC$ is a maximal isotropic subspace $\dim_\RR(L)=m-1$, and $e_1,e_2\in \Gamma(\TT M)$, satisfy
	\begin{align*}
	\langle e_1,e_1\rangle=0,\quad \langle e_2,e_2\rangle=0,\quad \langle e_1,e_2\rangle=1/2.
	\end{align*}
\end{definition}
A Sekiya quadruple can be associated to a generalised almost contact structure: Let $L$ represent the $+i$-eigenbundle of $\Phi$, and $e_1,e_2$ specify the $\pm \lambda$ eigenbundles respectively.

It is clear that the pairs $(e_1,e_2,\lambda)$ and $(e_2,e_1,-\lambda)$ describe the same generalised almost contact structure.  In particular if $\lambda=0$, $\dim(\ker(\Phi))=2$, and there is a $\sfO(1,1)$ freedom in the choice of $e_1,e_2$. 

A generalised almost contact structure on $M$ can be constructed from an $\sfS^1$-invariant generalised almost contact structure on a principal circle bundle $P(M,\pi,\sfU(1))$.   Consider a principal bundle $P(M,\pi,\sfU(1))$ over an odd-dimensional manifold $M$.  Let $\cU=\{U_\alpha \}$ denote a good cover of $M$, and $\pi^{-1}(U_\alpha)=U_\alpha\times\sfS^1$ a cover for $P$.  Take local coordinates $(x,t_\alpha)$, $x\in U_\alpha\cap U_\beta $,  $t_\alpha\in \sfS^1$. We have two set of coordinates on $\pi^{-1}(U_\alpha\cap U_\beta)$ $(x,t_\alpha)$ and $(x,t_\beta)$.  The coordinates are related by $t_\alpha=g_{\alpha\beta}t_\beta$, where $g_{\alpha\beta}\in \sfU(1)$ are transition functions.  Choose an $\sfS^1$-invariant connection $\cA$, given locally by $\cA_\alpha=dt_\alpha+A(x)$, where $A\in\Omega^1(M)$, and on $x\in U_{\alpha\beta}$ $\cA_\alpha=\cA_\beta-id\log g_{\alpha\beta}$.  Assume that there is an $\sfS^1$-invariant generalised almost complex structure $\cJ_{\text{inv}}\in\End(\TT P)$.  A choice of connection induces a decomposition of $S^1$-invariant sections $\Gamma(\TT P)=\Gamma(\TT M)\oplus C^\infty(M)\oplus C^\infty(M)$, with a section $v+\xi+f\partial_t+g\cA$, for $v\in \Gamma(TM)$, $f,g\in C^\infty(M)$, and $\xi\in\Gamma(T^*M)$.  This gives the decomposition:
\begin{equation}\label{genSekiya}
\cJ_{\text{inv}}=\left(
\begin{matrix}
\Phi & \mu e_1& \mu e_2\\
-2\mu\langle e_2,\cdot\rangle & -\lambda & 0\\ 
-2\mu\langle e_1,\cdot\rangle & 0 & \lambda
\end{matrix}
\right),
\end{equation}
where $\Phi\in\mathsf{End}(\TT M)$, $\mu=\sqrt{1+\lambda^2}$, $\lambda\in C^\infty(M)$, $e_1,e_2\in\TT M$.  The properties of a Sekiya quadruple $(\Phi,e_1,e_2,\lambda)$ follow from $\cJ^2_{\text{inv}}=-\text{id}$ and $\cJ^*_{\text{inv}}=-\cJ_{\text{inv}}$.  
The clutching construction for global sections $(X,f,g,\xi)$ involve transition functions $g_{\alpha\beta}\in \sfU(1)$.  The global sections can be viewed as a $(B,b,a)$-transformation with $B=b=0$, $a=\cA$, and twists $H_3=H_2=0$, and $F=d\cA$.  The choice of transition functions $g_{\alpha\beta}$ are not unique: the choice $g_{\alpha\beta}$ can be replaced with $g'_{\alpha\beta}=h_{\alpha\gamma}g_{\gamma\delta}h^{-1}_{\delta\beta}$, for $h_{\alpha\beta}\in\sfU(1)$, describing the same bundle.   Replacing $g_{\alpha\beta}$ with $g'_{\alpha\beta}$ corresponds to q $(0,0,a)$-gauge transformation, describing the decomposition with respect a connection $\cA'=\cA+a$, with $da=0$.    

\begin{definition}[\cite{Ald13}]
	A \emph{mixed pair} $(\varphi,\psi,e_1,e_2)$ consists of two polyforms $\varphi,\psi\in\Omega^{\text{ev/od}}(M)\otimes\CC$, and s choice of two sections $e_1,e_2\in\Gamma(\TT M)$ satisfying
	\begin{subequations}\label{mixedpairdef}
	\begin{align}
	(\varphi,\bar{\varphi})_{m-1}\neq 0,\quad (\psi,\bar{\psi})_{m-1}\neq &0,\quad ((\varphi,\psi),(\bar{\varphi},\bar{\psi}))_{m-1}\neq 0\label{mixpairing}\\
	e_1\cdot\psi=0,\quad \mu e_1\cdot\varphi=(1+i\lambda)\psi,\quad& e_2\cdot\varphi=0,\quad \mu e_2\cdot\psi=(1-i\lambda)\varphi,\label{mixcompat}
	\end{align} 	
	\end{subequations}
	where $\mu=\sqrt{1+\lambda^2}$, $\lambda\in C^\infty(M)$. 
\end{definition}

\begin{remark}
	The definition of mixed pair given here differs slightly from that given in \cite{Ald13}, which is valid for $\lambda=0$ only.  
\end{remark}

Given a nowhere zero function $f\in C^\infty(M)$, the pair $(f\varphi,f\psi)$ satisfies the equations \eqref{mixedpairdef} for fixed $(e_1,e_2,\lambda)$. 

A generalised almost contact structure can be described using a mixed pair. Given an almost contact structure $(L,e_1,e_2,\lambda)$, the identifications $\text{Ann}(\psi)=L\oplus \CC e_1$, and $\text{Ann}(\varphi)=L\oplus \CC e_2$, allow $L$ to be recovered as the intersection of the annihilator bundles of $(\varphi,\psi)$.

There is a local correspondence between generalised almost contact structures $(L,e_1,e_2,\lambda)$ and a conformal class of mixed pairs $(\varphi,\psi)$.

The definition of mixed pairs is motivated by the decomposition of a pure spinor, $\rho_\cJ$ (associated to an $\sfS^1$-invariant generalised complex structure $\cJ_{\text{inv}}$ on $M\times \sfS^1$ see \cite{Ald13}), into a mixed pair $(\varphi,\psi)$ associated to a Sekiya quadruple on $M$: $\rho_{\cJ}\rightarrow \varphi+idt\wedge \psi$. The pure spinor condition $(\varphi_\cJ,\bar{\varphi}_\cJ)_{M\times \sfS^1}\neq 0$ gives \eqref{mixpairing}.  Note that $\cJ|_{\sfS^1}(v_j)=iv_j$ ($j=1,2$) for $v_1=(\mu e_1,i-\lambda,0)$, $v_2=(\mu e_2,0,i+\lambda)$.  This implies that $v_j\cdot\varphi|_{\sfS^1}=0$, and gives \eqref{mixcompat}.

\begin{definition}
	A mixed pair $(\varphi,\psi)$ is \emph{$(H_3,H_2,F)$-involutive} if there exists a $V=(X,f,g,\xi)\in \Gamma(E)$ such that
	\begin{equation}
	d_{H_3,H_2,F}(\varphi,\psi)=V\cdot(\varphi,\psi),
	\end{equation}
	where 
	\begin{align*}
	d_{H_3,H_2,F}(\varphi,\psi):=&(d\varphi+H_3\wedge \varphi+F\wedge \psi,H_2\wedge\varphi-d\psi-H_3\wedge \psi),\\
	V\cdot(\varphi,\psi)=&(X,f,g,\xi)\cdot(\varphi,\psi):=(\iota_X\varphi+\xi\wedge \varphi+f\psi,g\varphi-\iota_X\psi-\xi\wedge\psi),
	\end{align*}
	and $(H_3,H_2,F)$ satisfy the Maurer-Cartan equations \eqref{MaurerCartan}.
\end{definition}

Given a non-zero function $h\in C^\infty(M)$, and mixed pair $(\varphi,\psi,e_1,e_2)$ satisfying $d_{H_3,H_2,F}(\varphi,\psi)=V\cdot (\varphi,\psi)$, the choice $(h\varphi,h\psi,e_1,e_2)$ satisfies 
\begin{align*}
d_{H_3,H_2,F}(h\varphi,h\psi)=V'\cdot (h\varphi,h\psi),\quad V'=(X,f,g,\xi-h^{-1}dh).
\end{align*}
Thus the $(H_3,H_2,F)$-involutive property is not dependent on the choice of $(\varphi,\psi)$ chosen to represent the almost contact structure $(L,e_1,e_2,\lambda)$.

\begin{definition}
	A \emph{$(H_2,H_2,F)$-generalised contact structure} $(L,e_1,e_2,\lambda)$ is a generalised almost contact structure $(L,e_1,e_2,\lambda)$ where all associated mixed pairs  $(\varphi,\psi,e_1,e_2)$ are $(H_3,H_2,F)$-involutive.  
\end{definition}

Let us briefly recall the Clifford product on $U\subset \wedge^\bullet T^*M\otimes \CC$ on the generalised tangent bundle $\TT M$:
\begin{align*}
(X,\xi)\cdot\rho=\iota_X\rho+\xi\wedge \rho,\quad (X,\xi)\cdot:\wedge^{\text{ev/odd}}T^*M\otimes \CC\rightarrow \wedge^{\text{odd/ev}}T^*M\otimes \CC,
\end{align*}
for $(X,\xi)\in\Gamma(TM)\oplus \Gamma(T^*M)$, and $\rho=U\subset\wedge^\bullet T^*M\otimes \CC$.
By Clifford multiplication by $U$, we obtain filtrations of the even and odd exterior forms (here
$2n$ is the real dimension of the manifold):
\begin{align*}
U =& U_0 < U_2 < \cdots < U_{2n} = \wedge^{\text{ev/odd}}T^*\otimes \CC,\\
L^* \cdot U =& U_1 < U_3 < \cdots < U_{2n-1} = \wedge^{\text{odd/ev}} T^*\otimes \CC,
\end{align*}
where ev/odd is chosen according to the parity of $U$ itself, and $U_k$ is defined as $CL^k\cdot U$, where $CL^k$ is spanned by products of not more than $k$ elements of $\TT M$ \cite{Gua04}. Note that,
using the inner product, we have the canonical isomorphism $L^* = ((T\oplus T^*) \otimes \CC)/L$, and so $U_1$ is
isomorphic to $L^*\otimes U_0$.  Theorem 3.38 of \cite{Gua04} shows that an almost Dirac structure defined by $\Ann(\rho)$ is Courant involutive iff $d(C^\infty(U_0))\subset C^\infty(U_1)$, which is equivalent to $d\rho=\iota_X\rho+\xi\wedge \rho$ for some $(X,\xi)\in(\Gamma(TM),\Gamma(T^*M))$.

A similar statement holds in the almost contact case.  Consider a mixed pair $(\varphi,\psi)$.  The definition requires that $(\varphi,\psi)\in(\wedge^{\text{ev/odd}}T^*M\otimes \CC,\wedge^{\text{odd/ev}}T^*M\otimes \CC)$. 
\begin{align*}
(X,f,g,\xi)\cdot(\varphi,\psi)&=(\iota_X\varphi+\xi\wedge \varphi+f\psi,g\varphi-\iota_X\psi-\xi\wedge\psi),\\
(X,f,g,\xi)\cdot:(\wedge^{\text{ev/odd}}T^*M\otimes\CC,\wedge^{\text{odd/ev}}&T^*M\otimes \CC)\rightarrow (\wedge^{\text{odd/ev}}T^*M\otimes \CC,\wedge^{\text{ev/odd}}T^*M\otimes \CC).
\end{align*} 
This gives a filtration on pairs $(\varphi,\psi)$:
\begin{align*}
V =& V_0 < V_2 < \cdots < V_{m+1} = (\wedge^{\text{ev/odd}}T^*\otimes \CC,\wedge^{\text{odd/ev}}T^*\otimes \CC),\\
L^* \cdot V =& V_1 < V_3 < \cdots < V_{m} = (\wedge^{\text{odd/ev}} T^*\otimes \CC,\wedge^{\text{ev/odd}}T^*\otimes \CC),
\end{align*}
where $\dim(M)=m$ is odd-dimensional and $\varphi,\psi$ are being viewed as pure spinors on a local trivialisation of $M\times \sfS^1$.
\begin{theorem}
	The annihilator bundle $\Ann(\varphi,\psi)$ of a $(H_3,H_2,F)$-involutive pair $(\varphi,\psi)$ is involutive under the $(H_3,H_2,F)$-contact Courant algebroid. 
\end{theorem}
\begin{proof}
	Let $(L,e_1,e_2,\lambda)$ be the generalised almost contact structure, and let $(\varphi,\psi)$ be a trivialization of representative of $(L,e_1,e_2,\lambda)$ over some open set.
	
	We show below that
	\begin{align*}
	\XX_1\circ_{H_3,H_2,F}\XX_2\cdot(\varphi,\psi) = -\XX_2 \cdot \XX_1 \cdot d_{H_3,H_2,F}(\varphi,\psi),
	\end{align*}
	for any sections $\XX_1,\XX_2\in \Ann(\varphi,\psi)$. Hence $\Ann(\varphi,\psi)$ is involutive if and only if $\XX_1 \cdot \XX_2 \cdot d_{H_3,H_2,F}(\varphi,\psi) = 0,\  \forall \XX_1,\XX_2 \in \Ann(\varphi,\psi)$, which is true if and only if $d_{H_3,H_2,F}(\varphi,\psi)$ is in $C^\infty(V_1)$, as elements of $V_k$ are precisely those which are annihilated
	by $k + 1$ elements in $\Ann(\varphi,\psi)$.
	
	Consider a section $(X,f,g,\xi)=\XX\in \Ann(\varphi,\psi)$, which satisfies
	\begin{align*}
	(X,f,g,\xi)\cdot(\varphi,\psi)=(\iota_X\varphi+\xi\wedge \varphi+f\psi,g\varphi-\iota_X\psi-\xi\wedge\psi)=0.
	\end{align*}
	Rearranging we have
	$\iota_X\varphi=-\xi\wedge\varphi-f\psi,\quad \iota_X\psi=g\varphi-\xi\wedge\psi$.
	
	\begin{align*}
	\iota_{[X_1,X_2]}\varphi=&[\cL_{X_1},\iota_{X_2}]\varphi=\cL_{X_1}\iota_{X_2}\varphi-\iota_{X_2}d\iota_{X_1}\varphi-\iota_{X_2}\iota_{X_1}d\varphi\\
	=&\cL_{X_1}(-\xi_2\wedge\varphi-f_2\psi)-\iota_{X_2}d(-\xi_1\wedge\varphi-f_1\psi)-\iota_{X_2}\iota_{X_1}d\varphi\\
	=&-\cL_{X_1}\xi_2\wedge\varphi-\xi_2\wedge d(-\xi_1\wedge\varphi-f_1\wedge\psi)-\xi_2\wedge\iota_{X_1}d\varphi-(\cL_{X_1}f_2)\psi-f_2\cL_{X_1}\psi\\
	&-\iota_{X_2}(-d\xi_1\wedge\varphi+\xi_1\wedge d\varphi-df_1\wedge\psi-f_1d\psi)-\iota_{X_2}\iota_{X_1}d\varphi\\
	=&(-\cL_{X_1}\xi_2+\iota_{X_2}d\xi_1+\xi_2\wedge d\xi_1)\wedge\varphi-\xi_2\wedge(\iota_{X_1}+\xi_1\wedge) d\varphi\\
	&+(\xi_2\wedge df_1-\iota_{X_1}df_2+\iota_{X_2}df_1)\wedge\psi+f_1\xi_2\wedge d\psi-f_2d(g_1\varphi-\xi_1\wedge\psi)-f_2d(g_1\varphi-\xi_1\wedge\psi)\\
	&-\iota_{X_2}(\xi_1\wedge d\varphi+\iota_{X_1}d\varphi)-df_1\wedge(g_2\varphi-\xi_2\wedge \psi)+f_1\iota_{X_2}d\psi\\
	=&-(\cL_{X_1}\xi_2-\iota_{X_2}d\xi_1+g_2df_1-f_2dg_1)\wedge\varphi-(\iota_{X_1}df_2-\iota_{X_2}df_1)\wedge\psi\\
	&+f_1(\iota_{X_2}+\xi_2\wedge) d\psi+f_2(\xi_1\wedge d\psi+g_1d\varphi)\\
	&-(\iota_{X_2}+\xi_2)\wedge(\iota_{X_1}+\xi_1\wedge) d\varphi.
	\end{align*}
	A similar calculation holds for $\iota_{[X_1,X_2]}\psi$.  Combining the results gives
	\begin{align*}
	\iota_{\XX_1\circ \XX_2}(\varphi,\psi)=\big(&-(\iota_{X_2}+\xi_2\wedge)(\iota_{X_1}+\xi_1\wedge)d\varphi-f_2g_1d\varphi-f_2(\iota_{X_1}+\xi_1\wedge)d\psi+f_1(\iota_{X_2}+\xi_2\wedge)d\psi,\\
	&(\iota_{X_2}+\xi_2\wedge)(\iota_{X_1}+\xi_1\wedge)d\psi+g_2f_1d\psi+g_1(\iota_{X_2}+\xi_2\wedge)d\varphi-g_2(\iota_{X_1}+\xi_1\wedge)d\varphi\big)\\
	=&-(X_2,f_2,g_2,\xi_2)\cdot((\iota_{X_1}+\xi_1\wedge)d\varphi-f_1d\psi,g_1d\varphi+(\iota_{X_1}+\xi_1\wedge)d\psi)\\
	=&-(X_2,f_2,g_2,\xi_2)\cdot(X_1,f_1,g_1,\xi_1)\cdot (d\varphi,-d\psi)=-\XX_2\cdot\XX_1\cdot d(\varphi,\psi).
	\end{align*}
	So the involutive corresponds to an involutive subspace.
		
A similar argument holds for the twisted case $\iota_{\XX_1\circ_{H_3,H_2,F}\XX_2}(\varphi,\psi)=-\XX_2\cdot\XX_1\cdot d_{H_3,H_2,F}(\varphi,\psi)$.
\end{proof}

Let us consider a generalised almost contact structure generated by a cosymplectic structure on $M$ and examine the integrability condition.  An almost cosymplectic structure is a pair $(\theta,\eta)\in (\Omega^2(M),\Omega(M))$ satisfying $\eta\wedge\theta^n\neq 0$.  From standard results in contact geometry, there exists a Reeb vector field $R\in \Gamma(TM)$ such that $\iota_R\eta=1$ and $\iota_R\theta=0$.  A mixed pair $(\varphi,\psi,e_1,e_2)$ can be given by
\begin{align*}
\varphi=e^{i\theta},\quad \psi=\eta\wedge e^{i\theta},\quad e_1=\eta,\quad e_2=R.
\end{align*} 
We have $d\varphi=id\theta\wedge \varphi$, and $d\psi=d\eta\wedge\varphi+id\theta\wedge \psi$.  If $d\theta=d\eta=0$ (a cosymplectic structure) then $d_{0,0,0}(\varphi,\psi)=0$.  If $d\eta=\theta$ (a contact 1-form $\eta$) then $d_{0,d\eta,0}(\varphi,\psi)=0$.  In fact the pair $(\theta,\eta)$ will form a $(0,d\eta,0)$-generalised contact structure if $d\theta=0$.  We conclude the following:
\begin{itemize}
	\item A cosymplectic structure defines a $(0,0,0)$-generalised contact structure.
	\item A contact 1-form $\eta$ defines a $(0,d\eta,0)$-generalised contact structure.
	\end{itemize}

It is possible to describe a non-coorientable contact structure arising from a $TM/D$-valued 1-form $\eta$. In the case of a non-trivial line bundle $TM/D$, $\eta$ is not globally defined. Local trivialisations $\eta_\alpha$ and $\eta_\beta$ are related using transition functions $g_{\alpha\beta}$.  If the line bundle on a compact manifold is of the form $TM/D\cong T\sfS^1$, for an $\sfS^1$-foliation then $\eta$ satisfies the conditions of a $(0,0,\eta)$-transformation (\eqref{CurvBba} and \eqref{CocycleBba}) with twists $(0,0,d\eta)$. 

\begin{remark}
	In \cite{Poo09} an almost generalised structure (defined with $e_1\in \Gamma(TM)$, $e_2\in \Gamma(T^*M)$, $\lambda=0$) is called a generalised structure if $L\oplus \CC e_1$ is involutive.  A \emph{strong} generalised contact structure is a generalised contact structure where $L\oplus \CC e_2$ is involutive.  In \cite{Ald13} a generalised normal contact structure is a generalised contact structure arising from an invariant generalised complex structure $\cJ$ on $M\times \RR$.  In both cases a contact form $\eta$ with $d\eta\neq 0$ does not give a strong generalised contact structure.    
\end{remark}

An essential property of the definition of a generalised (almost) contact structure is that it is compatible with the $(B,b,a)$-symmetries.  The action on a Dirac pair $(\varphi,\psi)$ is: 
\begin{equation}\label{spinorsymm}
e^{(B,b,a)}(\varphi,\psi)=\left(e^{-B}\varphi+ae^{-B}\psi-\frac{1}{2}abe^{-B}\varphi,e^{-B}\psi-be^{-B}\varphi-\frac{1}{2}bae^{-B}\psi\right)=(\varphi',\psi'),
\end{equation}
where $\wedge$ has been omitted.  The action preserves the pairing,
\begin{align*}
(e^{(B,b,a)}(\varphi_1,\psi_1),e^{(B,b,a)}(\varphi_2,\psi_2))_M=((\varphi_1,\psi_1),(\varphi_2,\psi_2))_M,
\end{align*}  
and satisfies
\begin{align}\label{twistint}
d_{H'_3,H'_2,F'}(e^{(B,b,a)}(\varphi,\psi))-V'\cdot(e^{(B,b,a)}(\varphi,\psi))=e^{(B,b,a)}(d_{H_3,H_2,F}(\varphi,\psi)-V\cdot(\varphi,\psi)),
\end{align}
where
\begin{align*}
H'_3=&H_3+dB+a\wedge H_2+b\wedge F+\frac{1}{2}(da\wedge b+a\wedge db),\\
H'_2=&H_2+db, \quad F'=F+da,\quad V'=e^{(B,b,a)}V.
\end{align*}
This shows that given a $(H_3,H_2,F)$-involutive mixed pair, $(\varphi,\psi)$, there exists a $(H'_3,H'_2,F')$-involutive Dirac pair $e^{(B,b,a)}(\varphi,\psi)$.

Below two standard examples of generalised contact structures (first appearing in \cite{Igl04} for $M\times \RR$) are presented as reductions of generalised complex structures on $\sfS^1\hookrightarrow P\rightarrow  M$. 

\begin{example}[Almost symplectic to almost cosymplectic structure]\label{cosymp}
	Let $\omega\in\Gamma(\wedge^2T^*P)$ be a symplectic form, where $\sfS^1\hookrightarrow P\rightarrow M$, and $\dim(M)=m=2n+1$.  Consider a connection specified by locally by $\cA_\alpha=dt_\alpha+A_\alpha\in\Omega^1(P_\alpha)$ ($\pi^{-1}(U_\alpha)=U_\alpha\times \sfS^1$ defining a cover of $P$), with curvature $F=d\cA\in H^2(M,\ZZ)$.  
	The symplectic form is $\sfS^1$-invariant if it admits the decomposition
	\begin{align*}
	\omega_\alpha=\theta_\alpha+\cA_\alpha\wedge\eta_\alpha\in\Omega^2(P_\alpha),
	\end{align*}
	where $\omega$ is globally defined, but $\eta$ is not if the bundle is not trivial, and satisfies the usual cocycle conditions.  
	\begin{align*}
	0\neq \omega^{n+1}=(\theta+\cA\wedge\eta)^{n+1}=\sum^{n+1}_{k=0}\theta^{n+1-k}(\cA\wedge\eta)^k=(n+1)\theta^{n}\wedge\cA\wedge\eta,
	\end{align*}
	 giving $\eta\wedge\theta^{n}\neq 0$.  From standard results in contact geometry, there exists a Reeb vector field $R\in \Gamma(TM)$ such that $\iota_R\eta=1$ and $\iota_R\theta=0$.  $\theta$ is non-degenerate on $\ker(\eta)$, and  $\phi(X):=\iota_X\theta$ ($X\in\ker(\eta)$).  The generalised complex structure $\cJ_\omega$ is reduced to a generalised contact structure $(\Phi,e_1,e_2,\lambda)$:
	\begin{align*}
	\cJ_\omega=\left(\begin{matrix}
	0 & -\omega^{-1}\\ \omega & 0
	\end{matrix}\right)\quad \Rightarrow&\quad \Phi=\left(\begin{matrix}
	0 & -\theta^{-1}\\ \theta & 0
	\end{matrix}\right),\quad e_1=\eta,\quad e_2=R,\quad \lambda=0,\\
	\rho_{\cJ_\omega}=e^{i\omega}\quad\Rightarrow&\quad \varphi= e^{i\theta},\quad \psi=\eta \wedge e^{i\theta},
	\end{align*}
	where $\cJ_\omega$ is written in the splitting $TP\oplus T^*P$, $\Phi$ in the splitting $\ker(\eta)\oplus\text{Ann}(R)$, and $\rho_{\cJ_\omega}$ is a pure spinor associated to $\cJ_\omega$.
\end{example}
Let us look at the integrability conditions:
\begin{align*}
d\varphi=id\theta\wedge\varphi,\quad d\psi=d\eta\wedge \varphi+id\theta\wedge\psi.
\end{align*}
If $d\theta=0$ then $d_{0,d\eta,0}(\varphi,\psi)=0$.  Noting that $\omega=\theta+\cA\wedge\eta$, $d\omega=d\theta+d\alpha\wedge\eta+\cA\wedge d\eta$ is not necessarily zero.  So a generalised $d_{0,d\eta,0}$-contact structure can arise from a pre-symplectic structure. 
  
\begin{example}[almost complex to almost contact structure]\label{contact}
	Consider an almost complex structure $J\in\mathsf{End}(TP)$, on $\sfS^1\hookrightarrow P\rightarrow M$, where $\dim{M}=2n+1$, and the the $\sfS^1$-bundle specified the connection $\cA$, (with $F=d\cA$) given locally by $\cA_\alpha=dt_\alpha+A_\alpha\in\Omega^1(P_\alpha)$.  Given local coordinates, $x$, for $M$, and $t$ for $\sfS^1$, the almost complex structure is $\sfS^1$-invariant if there exists the decomposition  
	\begin{align*}
	J^I_K\partial_I\otimes dx^K= (\phi'^i_k-R^iA_k)\partial_i\otimes dx^k+R^i\partial_i\otimes\cA+\alpha_k\kappa\otimes dx^k.
	\end{align*}  
	Letting $\phi=\phi'-R\otimes A$,  the conditions $J^*=-J$ and $J^2=-1$ give 
	\begin{align*}
	\iota_R\alpha=1,\quad \Phi(R)=0=\phi^*(\alpha),\quad \phi^2(X)=-X+(\iota_X\alpha)R.
	\end{align*} 
	The generalised almost complex structure $\cJ_J$ reduces to a generalised almost contact structure $(\Phi,e_1,e_2,\lambda)$:
	\begin{align*}
	\cJ_{J}=\left(\begin{matrix}
	J & 0\\ 0 &-J^*
	\end{matrix}\right)\quad\Rightarrow&\quad \Phi=\left(\begin{matrix}
	\phi & 0\\ 0 &-\phi^*
	\end{matrix}\right),\quad e_1=\alpha,\quad e_2=R,\quad \lambda=0,\\
	\rho_{J}=\Omega_J\quad \Rightarrow &\quad \varphi=\Omega_\phi,\quad \psi=\alpha\wedge \Omega_\phi,
	\end{align*}
	where $\Omega_J\in \Omega^{2n+2,0}$ is the decomposable top form giving the pure spinor describing $\cJ_J$, and $\alpha\wedge \Omega_\phi:=\Omega_J$.
\end{example}
Let us look at the integrability conditions:
\begin{align*}
d\varphi=d\Omega_\phi,\quad d\psi=d\alpha\wedge \varphi+\alpha\wedge d\Omega_\phi.
\end{align*}
We require that $d\Omega_\phi=0$.  In this case we have $d_{0,d\alpha,0}(\varphi,\psi)=0$.

\subsection{Deformations of generalised contact structures}~\\
The $(B,b,a)$-transformations described in Section \ref{symmetries}, provide deformations of generalised contact structures.  While $B$-transformations have been studied before, the $(b,a)$-transformations have not been incorporated. The $\mathcal{K}_\pm(\kappa)$ ($\kappa\in\Gamma(T^*M)$) symmetries introduced by Sekiya in \cite{Sek12} are equivalent to the more geometrically natural $(b,a)$-transformations.

A Sekiya quadruple $(\Phi,e_1,e_2,\lambda)$, can be deformed to give another Sekiya quadruple $(\Phi',e'_1,e'_2,\lambda')$: 
	\begin{subequations}\label{Seksymm}
		\begin{align}
		\Phi'(v)=&e^B\Phi e^{-B}(v)-\langle v,a\rangle\Phi(b)-\langle v,b\rangle\Phi(a)+2\mu\langle v,a\rangle e_1+2\mu\langle v,b\rangle e_2\nonumber \\
		&+\langle e^B\Phi(a),v\rangle b-2\mu \langle e_1,a\rangle\langle v,a\rangle b-2\mu\langle e_2,a\rangle \langle v,b\rangle b-2\mu\langle e^B e_2,v\rangle b\label{Phiprime}\\
		&+\langle e^B\Phi(b),v\rangle a-2\mu \langle e_2,b\rangle\langle v,b\rangle a-2\mu\langle e_1,b\rangle \langle v,a\rangle a-2\mu\langle e^B e_1,v\rangle a. \nonumber\\
		&+\langle a,\Phi(b)\rangle\langle a,v\rangle b+\langle b,\Phi(a)\rangle\langle b,v\rangle a.\nonumber\\
		\mu'e'_1=&\mu e^Be_1-\lambda' b+\mu\langle e_1,b\rangle a+\mu\langle e_1,a\rangle b-e^B\Phi(b)\label{e1prime}\\
		\mu'e'_2=&\mu e^Be_2+\lambda' a+\mu\langle e_2,a\rangle b+\mu\langle e_2,b\rangle a-e^B\Phi(a)\label{e2prime}\\
		\lambda'=&\lambda+2\langle b,\Phi(a)\rangle+2\mu\langle e_1,a\rangle-2\mu\langle e_2,b\rangle\label{lambdaprime}
		\end{align}
	\end{subequations}
where $v=(X,\xi)\in\Gamma(\TT M)$.

This follows immediately from considering $e^{(B,b,a)}$ transformation on $\cJ_{\text{inv}}$, via
\begin{align*}
e^{(B,b,a)}\cJ_{\text{inv}}e^{(-B,-b,-a)}=\left(\begin{matrix} \Phi' & \mu' e'_1 & \mu' e'_2\\ -2\mu'\langle e'_2,\cdot\rangle & -\lambda' & 0\\
-2\mu'\langle e'_1\cdot\rangle & 0 & \lambda'\end{matrix}\right),
\end{align*}
where $\cJ_{\text{inv}}$ is the $\sfS^1$-invariant generalised complex structure associated to $(\Phi,e_1,e_2,\lambda)$ by \eqref{genSekiya}, and noting that $e^{(B,b,a)}e^{(-B,-b,-a)}\XX=\XX$. 

\begin{proposition}\label{twistingtrans}
	The transformation $e^{(B,b,a)}$ given by \eqref{Symm} maps a $(H_3,H_2,F)$-involutive generalised contact structure $(\Phi,e_1,e_2,\lambda)$ to a $(H'_3,H'_2,F')$-involutive generalised contact structure $(\Phi',e'_1,e'_2,\lambda')$ given by \eqref{Seksymm},
	where
	\begin{align*}
	H'_3=H_3+dB+a\wedge H_2+b\wedge F+\frac{1}{2}(da\wedge b+a\wedge db),\ H'_2=H_2+db,\ F'=F+da.
	\end{align*}
\end{proposition}
\begin{proof}
	This follows directly from the mixed pair description and the fact that the transformation $e^{(B,b,a)}$ preserves the pairing and \eqref{twistint}.
\end{proof}

\begin{example}[Symmetries]
	Any generalised almost contact structure gives a family of generalised almost contact structures using $(B,b,a)$-transformations.  Deforming Example \ref{cosymp} gives
	\begin{align*}
	&\varphi'=\left(1+a\wedge \eta-\frac{1}{2}a\wedge b\right)\wedge e^{-B+i\theta},\quad \psi'=\left(\eta-b-\frac{1}{2}b\wedge a\wedge\eta\right)\wedge e^{-B+i\theta}\\
	&\mu'e'_1=e^B\wedge\eta-\lambda'b+\iota_{\rho^*(b)}\theta^{-1}+\iota_{\iota_{\rho^*(b)}\theta^{-1}}B\\
	&\mu'e'_2=R+\iota_R B+\lambda' a+\frac{1}{2}(\iota_R a)b+\frac{1}{2}(\iota_R b)a+\iota_{\rho(a)}\theta^{-1}+\iota_{\iota_{\rho(a)}\theta^{-1}}B,\\
	&\lambda'=\iota_R b-\iota_{\rho^*(b)}\iota_{\rho^*(a)}\theta^{-1},
	\end{align*}
	where $\rho^*(a):T^*M\rightarrow \text{Ann}(R)$ is the projection. 
	
	Deforming Example \ref{contact} we get  
	\begin{align*}
	&\varphi'=\left(1+a\wedge\alpha-\frac{1}{2}a\wedge b\right)\wedge e^{-B}\wedge\Omega_\phi,\quad \psi'=\left(\alpha-b-\frac{1}{2}b\wedge a\right)\wedge e^{-B}\wedge\Omega_\phi,\\
	&\mu' e'_1=R+\iota_R B-\lambda' b+\frac{1}{2}(\iota_R b)a+\frac{1}{2}(\iota_R a)b+\phi^*(\rho(b)),\\
	&\mu'e'_2=\alpha+\lambda'a+\phi^*(\rho(a)),\\
	&\lambda'=\iota_R a.
	\end{align*}
\end{example}

\begin{remark}
	These examples show that $(b,a)$-transformations can change $\lambda$.  The $(b,a)$-transformation can be interpreted geometrically as twisting the $\sfS^1$-bundle.  The correspondence between $a$-transformations and twisting comes from the discussion on page 11.  The description on page 8 shows that the splitting of invariant sections of $TP$ and $T^*P$ correspond to an $a$-transformation and a dual, the $b$-transformation. 
\end{remark}

\begin{example}[Products \cite{Gom13}] Let $M=M_1\times M_2$, with projections $pr_i:M\rightarrow M_i$. If $(L_1,e_1,e_2)$ is a generalised almost contact structure on $M_1$ and $L_2$ is a generalised almost complex structure on $M_2$, then $(pr^*_1L_1\oplus pr^*_2L_2,pr^*_1e_1,pr^*_1e_2)$ is a generalised almost contact structure on $M$.  	
\end{example}

There are manifolds which admit generalised contact structures, but not contact structures.  A class of examples come from $\sfS^1$-bundles of nilmanifolds.  A nilmanifold is a homogenous space $M=\sfG/\Gamma$, where $\sfG$ is a simply connected nilpotent real Lie group and $\Gamma$ is a lattice of maximal rank in $\sfG$.  For the associated generalised complex structures on nilmanifolds, see \cite{Cav04}.  The structure of a particular nilpotent Lie algebra can be given by specified listing exterior derivatives of the elements of a Malcev basis, as an $n$-tuple of two-forms $d\epsilon_k=\sum c^{ij}_k\epsilon_i\epsilon_j$, (henceforth $\wedge$ is omitted, so that $\epsilon_i\wedge \epsilon_j=\epsilon_i\epsilon_j$).

\begin{example}[$(0,0,12,13,14+23,34+52)\times \sfS^1$]
	Specify a 6-dimensional nilmanifold via the coframe $\{\epsilon_i \}$, $i=1,\dots,6$ satisfying: 
	\begin{align*}
	d\epsilon_1=0,\ d\epsilon_2=0,\ d\epsilon_3=\epsilon_1\epsilon_2,\ d\epsilon_4=\epsilon_1\epsilon_3,\ d\epsilon_5=\epsilon_1\epsilon_4+\epsilon_2\epsilon_3,\ d\epsilon_6=\epsilon_3\epsilon_4+\epsilon_5\epsilon_2.
	\end{align*}
	Let $E=M\times \sfS^1$, where $M$ is the nilmanifold specified by ($0,0,12,13,14+23,34+52$), and $\sfS^1$ is parameterized by $t$.  The one-form $dt$ gives a flat connection on $S^1$.  Define $\eta=\pi^*dt$. Let $R=\pi^*\partial_t$ be the corresponding Reeb vector field:  
	\begin{align*}
	\Omega=&\epsilon_1+i\epsilon_2,\\
	B=&\epsilon_2\epsilon_6-\epsilon_3\epsilon_5+\epsilon_3\epsilon_6-\epsilon_4\epsilon_5,\\
	\omega&=\epsilon_3\epsilon_6+\epsilon_4\epsilon_5,\\
	\varphi=&e^{B+i\omega}\Omega,\quad \psi=\eta e^{B+i\omega}\Omega,\quad e_1=\eta,\ e_2=R,\ \lambda=0.
	\end{align*}
\end{example}

\begin{example}[$\sfS^1$-bundles on nilmanifolds]
	There are manifolds which have no symplectic or complex structures, but do have generalised complex structures.  
	In \cite{Cav04} generalised complex structures are constructed on nilmanifolds which do not admit symplectic, or complex structures.  Each of these examples define a generalised complex structure via a pure spinor $\rho=\Omega\wedge e^{B+i\omega}$. This construction can be modified to find generalised contact structures which do not admit contact structures.  Take  $\sfS^1\hookrightarrow E\rightarrow M$.  Choose an $\sfS^1$-invariant connection $\cA$.  Define a vector field $R$ such that $\iota_R \cA=1$.  Take the generalised complex structure described by the pure spinor $\rho=\Omega\wedge e^{i\omega+B}$. The corresponding mixed pair is
	\begin{align*}
	\varphi=e^{B+i\omega}\Omega,\quad \psi=\cA e^{B+i\omega}\Omega,\quad e_1=\cA,\ e_2=\kappa,\ \lambda=0.
	\end{align*}
\end{example}

\begin{example}
	Consider $\RR^5$, described using coordinates $\{t,z_1,z_2 \}$ where $z_1,z_2$ are standard coordinates in $\CC^2\cong \RR^4$.  A generalised complex structure is defined by the pure spinor $\rho=z_1+dz_1 dz_2$. When $z_1=0$, $\varphi=dz_1 dz_2$ defines a standard complex structure, whereas $z_1\neq 0$, $\varphi$ defines a $B$-symplectic structure since
	$\rho=z_1\exp(dz_1 dz_2/z_1)$.  A generalised contact structure is given by 
	\begin{align*}\varphi=z_1+dz_1 dz_2,\quad \psi=dt (z_1+dz_1 dz_2),\quad e_1=dt,\ e_2=\partial_t.
	\end{align*}
\end{example}

\section{Generalised coK\"ahler geometry}\label{genSasakigeom}
Generalised geometric structures are of great interest in string theory, due to the fact that T-duality is associated to $\mathfrak{so}(T\oplus T^*)=\End(T)\oplus \wedge^2 T^*\oplus \wedge^2 T$.  In particular, the generalised metric incorporates the Riemannian metric, $\sfg$, and $B$-field associated with the Neveu-Schwarz flux $H$, in the bosonic sector of supergravity.  Generalised K\"ahler structures are equivalent to bi-hermitian structures and are the most general geometry associated to 2-dimensional target space models with $\cN=(2,2)$ supersymmetry \cite{Gat84}.  

CoK\"ahler structures are the odd-dimensional counterpart to K\"ahler structures.  The relationship between K\"ahler and coK\"ahler structures is described in \cite{Li08,Baz14}. Li gave a structure result for compact coK\"ahler manifolds stating that such a manifold is always a K\"ahler mapping torus. The coK\"ahler structure on an odd-dimensional manifold $M$ can be associated to a K\"ahler structure on an $\sfS^1$-bundle (using a symplectomorphism) \cite{Li08}.  Further results on coK\"ahler structures were given in \cite{Baz14}.

Generalised coK\"ahler structures have appeared in the literature before \cite{Gom15}.  The definition given in \cite{Gom15} deals with generalised K\"ahler structures on $M_1\times M_2$, and the definition is compatible with $B$-transformations.  In this note we will consider the case where $M_2=\sfS^1$, but will not restrict to product manifolds, instead considering principal circle bundles, and the definition is compatible with the full $(B,b,a)$-transformations. Generalised coK\"ahler structures will be presented as $\sfS^1$-invariant reductions of generalised K\"ahler structures.

\begin{remark}
	Generalised K\"ahler structures play an important role in string theory.  In \cite{Gat84} generalised K\"ahler structures (written as a bi-Hermitian structures) appear in the study of $\cN=(2,2)$ non-linear sigma models with torsion.  The torsion arises from the connections $\nabla^\pm=\nabla^{\text{LC}}\pm g^{-1}H$, where $\nabla^{\text{LC}}$ is the Levi-Civita connection. Abelian T-duality can be carried out when the metric has an $\sfS^1$-isometry, and the T-duality procedure involves Kaluza-Klein reduction.  T-duality is most interesting when the $\sfS^1$-isometry corresponds to a topologically non-trivial $\sfS^1$-bundle.  In this case there is an interesting relationship between topology and $H$-flux \cite{Bou04a,Bou04b}.  The study of $\sfS^1$-reductions of generalised K\"ahler structures is interesting in this context.
\end{remark}

\subsection{Generalised metric structure}
The inner product \eqref{innerprod} is non-degenerate and a generalised contact metric can be constructed using maximally isotropic subspaces 
\begin{align*}
\cG(\XX_1,\XX_2)=\langle \XX_1,\XX_2\rangle|_{C_+}-\langle \XX_1,\XX_2\rangle|_{C_-},
\end{align*}
for $\XX_1,\XX_2\in\Gamma(E)\cong \Gamma(TM)\oplus C^\infty(M)\oplus C^\infty(M)\oplus \Gamma(T^*M)$,
as in the case of generalised geometry on $\TT M$, see Section \ref{gengeomstructures}.  In this case we have 
\begin{align}\label{GenContactMetric}
C_\pm=\{(X,f,g,\xi)\in\Gamma(E):g=\pm fh^2,\ \xi=\pm \sfg(X,\cdot)\}
\end{align}
for some $h\in C^\infty(M)$.  This satisfies
\begin{align*}
\langle(X,f,\pm fh^2,\pm\sfg(X,\cdot)),(X,f,\pm fh^2,\pm\sfg(X,\cdot))\rangle=\pm\sfg(X,X)\pm f^2h^2,
\end{align*}
verifying that $C_\pm$ describe the maximal positive/negative definite subbundles.  

As $\langle\cdot,\cdot\rangle$ is invariant under $(B,b,a)$-transformations, the subbundles $C_\pm$ defining a generalised metric $\cG$ can be transformed to $e^{(B,b,a)}C_\pm$ defining a generalised metric $\cG'=e^{(B,b,a)}\cG e^{-(B,b,a)}$.  The maximal subspaces are given by
\begin{align}\label{GenContactMetricBba}
C_\pm=\{(X,\xi,f,g):&\xi=\sfg(X,\cdot)+B(X,\cdot)-fb-fh^2a-2\langle X,b\rangle a,\\
&g=fh^2+2\langle X,b\rangle+2h^2\langle X,a\rangle+\langle X,b\rangle\langle X,a\rangle \}.\nonumber
\end{align}
All subspaces $C_\pm$ can be described in the form \eqref{GenContactMetricBba} for some choice of $(\sfg,h,B,b,a)$.
\begin{definition}
	A \emph{generalised coK\"ahler structure} on an odd-dimensional manifold $M$, consists of two generalised $(H_3,H_2,F)$-contact structures $(L_1,e^{(1)}_1,e^{(1)}_2,\lambda_1)$ and $(L_2,e^{(2)}_2,e^{(2)}_2,\lambda_2)$ whose associated Sekiya quadruples $\cJ_1=\{\Phi_1,e^{(1)}_1,e^{(1)}_2,\lambda_1\}$ and $\cJ_2=\{\Phi_2,e^{(2)}_1,e^{(2)}_2,\lambda_2\}$ give a generalised K\"ahler structure. 
\end{definition}

	The commuting condition $\cJ_1\cJ_2=\cJ_2\cJ_1$ places the following restrictions on the Sekiya quadruples  $(\Phi_1,e^{(1)}_1,e^{(1)}_2,\lambda_1)$, and $(\Phi_2,e^{(2)}_1,e^{(2)}_2,\lambda_2)$, satisfying:
	\begin{align*}
	\CC e^{(1)}_1\oplus \CC e^{(1)}_2=&\CC e^{(2)}_1\oplus \CC e^{(2)}_2:=\cE;\\
	\Phi_1\Phi_2(v)=&\Phi_2\Phi_1(v)\quad\forall v\in\cE^\perp\\
	e^{(1)}_1=e^{(2)}_1,\ e^{(1)}_2=e^{(2)}_2,\ {\text{when }}\lambda_1=\lambda_2\quad \text{or}&\quad e^{(1)}_1=e^{(2)}_2,\ e^{(1)}_2=e^{(2)}_1,\ \forall\lambda_1,\lambda_2\in\RR. 
	\end{align*}

\begin{remark}
	If $\lambda_1=\lambda_2=0$ then there is a $\sfO(1,1)$-freedom in the description, and $e^{(1)}_1=e^{(2)}_2$, and $e^{(1)}_2=e^{(2)}_1$ can be chosen.    
\end{remark}

\begin{example}[coK\"ahler] A \emph{coK\"ahler structure} on an odd-dimensional manifold $M$, is given by the quadruple $(J,R,\eta,\sfg)$, where $(J,R,\eta)$ is an almost contact structure and $\sfg$ is a Riemannian metric satisfying $\sfg(JX,JY)=g(X,Y)-\eta(X)\eta(Y)$ for all $X,Y\in\Gamma(TM)$, and the integrability conditions $[J,J]=0,\quad d\omega=d\eta=0$,	
where $\omega(X,Y):=g(JX,Y)\in \Omega^2(M)$.  This defines a generalised coK\"ahler structure with $(\lambda=H_3=H_2=F=0)$:
\begin{align*}
\varphi_1=e^{i\omega},\ \psi_1=\eta\wedge e^{i\omega},\ e^{(1)}_1=\eta,\ e^{(1)}_2=R,\quad \varphi_2=\Omega_J,\ \psi_2=\eta\wedge \Omega_J,\ e^{(1)}_1=\eta,\ e^{(1)}_2=R.
\end{align*}
\end{example}

\begin{example}[Generalised K\"ahler to generalised coKahler]\label{GenSasaki}
	Consider the reduction of a generalised K\"ahler structure to produce a generalised coK\"ahler structure.  It was shown in Example \ref{cosymp} (Example \ref{contact}) that the reduction of a symplectic (complex) structure (over the same $\sfS^1$-bundle) gives
	\begin{align*}
	\cJ_\omega{}_{\text{inv}}=\left(\begin{matrix}
	0 & -\omega^{-1} & 0 & \eta \\ \omega & 0 & R & 0\\
	0 & -2\langle\eta,\cdot\rangle & 0 & 0\\
	-2\langle R,\cdot\rangle & 0 & 0 & 0 
	\end{matrix}\right),\quad \cJ_\Omega{}_{\text{inv}}=\left(\begin{matrix}
	-J & 0 & -R & 0\\ 0 &J^* & 0 & \eta\\
	2\langle R,
	\cdot\rangle & 0 & 0 & 0\\
	0 & -2\langle \eta,\cdot\rangle & 0 & 0  
	\end{matrix}\right),	
	\end{align*}
	where $\omega$ and $J$ are non-degenerate on $D$.  The condition that $-\cJ_\omega{}_{\text{inv}}\cJ_\Omega{}_{\text{inv}}=\cG$ for 
	\begin{align*}
	\cG=\left( \begin{matrix}
	0 & \sfg^{-1} & 0 & 0\\ 
	\sfg & 0 & 0 & 0\\
	0 & 0 & 0 & 1\\
	0 & 0 & 1 & 0
	\end{matrix}\right),
	\end{align*}
	requires that $\omega(JX,Y)=\sfg(X,Y)$, giving a transverse K\"ahler structure, $\langle R,\eta\rangle=\frac{1}{2}$, $\langle R,R\rangle=0=\langle \eta,\eta\rangle$.  
	
	The almost generalised complex structures $\cJ_\omega{}_{inv}$ and $\cJ_\Omega{}_{inv}$ will define a generalised coK\"ahler structure when $H_2=d\alpha=d\eta$ (see Examples \ref{cosymp} and \ref{contact} for notation).
\end{example}
\begin{example}[Twisted generalised coK\"ahler]
	It is clear from Example \ref{GenSasaki} that the reduction of a generalised K\"ahler structure can produce a generalised coK\"ahler structure.  It is possible to deform any generalised coK\"ahler structure $(\cJ_\omega{}_{\text{inv}},\cJ_\Omega{}_{\text{inv}},\cG)$ to get another:
	 \begin{align*}
	 (e^{(B,b,a)}\cJ_\Omega{}_{\text{inv}}e^{-(B,b,a)},e^{(B,b,a)}\cJ_\omega{}_{\text{inv}}e^{-(B,b,a)},e^{(B,b,a)}\cG e^{-(B,b,a)}).  
	 \end{align*}
\end{example}

\begin{definition}
	A \emph{generalised almost coK\"ahler-Einstein structure} on an odd-dimensional manifold $M$ (with $m=\dim(M)$) is described by two mixed pairs $(\varphi_1,\psi_1)$ and $(\varphi_2,\psi_2)$ satisfying
	\begin{align*}
	((\varphi_1,\psi_1),(\bar{\varphi}_1,\bar{\psi}_1))_M=c((\varphi_2,\psi_2),(\bar{\varphi}_2,\bar{\psi}_2))_M,
	\end{align*}
	where $c\in \RR$ can be scaled to $+1$ or $-1$ by scaling $(\varphi_1,\psi_1)$.  A generalised almost coK\"ahler-Einstein structure is an generalised coK\"ahler-Einstein structure if $(\varphi_1,\psi_1,e_1,e_2)$ and $(\varphi_2,\psi_2,e_1,e_2)$ are generalised $(H_3,H_2,F)$-contact structures.
\end{definition}

\begin{example}[$\sfS^1$-invariant generalised Calabi-Yau]\label{ReductionCalabiYau}
	Let $N=M\times\sfS^1$ be an even dimensional manifold with an $\sfS^1$-invariant generalised Calabi-Yau structure $(\rho_1,\rho_2)$.  The decompositions $\rho_j=\varphi_j+idt\psi_j$ ($j=1,2$) defines a generalised coK\"ahler-Einstein structure:  $(\varphi_1,\psi_1,e^{(1)}_1=\partial_t,e^{(1)}_2=dt)$ and $(\varphi_2,\psi_2,e^{(1)}_1=\partial_t,e^{(1)}_2=dt)$, where $\lambda=H_3=H_2=F=0$.  
\end{example}

\begin{example}[coK\"ahler-Einstein]
	A coK\"ahler-Einstein structure on an odd-dimensional manifold $M$, is a Ricci-flat coK\"ahler structure.  A coK\"ahler structure has an associated cosymplectic structure $(\eta,\theta)$.  Consider $N=M\times \sfS^1$, with $\sfS^1$ parameterised by $t$, and $pr_1(N)=M$.  Let $\omega=dt\wedge pr^*_1\eta+pr^*_1\theta$, and $\sfg_N=pr^*_1\sfg+(dt)^2$.  This defines a Calabi-Yau structure on $N$.  A Calabi-Yau structure defines a generalised Calabi-Yau structure (Example \ref{CalabiYauEx}).  Using the reduction procedure (Example \ref{ReductionCalabiYau}) we get a generalised $(0,0,0)$-coK\"ahler-Einstein structure.
\end{example}

\begin{example}
	A $(B,b,a)$-transformation maps an involutive mixed pair to another involutive mixed pair, preserving the length.  It follows that a generalised coK\"ahler(-Einstein) structure, $((\varphi_1,\psi_1),(\varphi_2,\psi_2))$, is mapped to another coK\"ahler(-Einstein) structure by a $(B,b,a)$-transformation, $((e^{(B,b,a)}\varphi_1,e^{(B,b,a)}\psi_1),(e^{(B,b,a)}\varphi_2,e^{(B,b,a)}\psi_2))$.	A generalised $(H_3,H_2,F)$-contact structure is mapped to generalised $(H_3+dB,H_2+db,F+da)$-contact structure.
\end{example}

\section{T-duality}\label{Tduality}
T-duality provides an isomorphism between between Courant algebroids defined on two torus bundles $\sfT^k\hookrightarrow E\rightarrow M$ and $\tilde{\sfT}^k\hookrightarrow \tilde{E}\rightarrow M$.  The topological aspects are described in \cite{Bou04a,Bou04b}, and the isomorphism of Courant algebroid structures in \cite{Cav10}.  The situation is described by the following diagram:
\begin{align*}
\xymatrix{
	& \ar[ld]^{p} (E\times_M \tilde{E},p^*H-\tilde{p}^*\tilde{H}) \ar[rd]^{\tilde{p}}& \\
	{e_1\circ_H e_2}  \ (\sfT^k\hookrightarrow E\rightarrow M,H)\ar[rd]^{\pi} & & (\tilde{\sfT}^k\hookrightarrow \tilde{E}\rightarrow M,\tilde{H})\ar[ld]^{\tilde{\pi}}\ \  \tilde{e_1}\circ_{\tilde{H}}\tilde{e_2} \\
	& M & 
}
\end{align*}
$E$ and $\tilde{E}$ are \emph{T-dual} if $p^*H-\tilde{p}^*\tilde{H}=d\cF$, for some $T^{2k}$-invariant 2-form on the correspondence space $\cF\in\Omega^2(E\times_M \tilde{E})$, such that $\cF:\mathfrak{t}^k_E\otimes \mathfrak{t}^k_{\tilde{E}}\rightarrow \RR$ is non-degenerate.  These conditions place restraints on $H$.  $H$ is \emph{admissible} if it satisfies \cite{Cav10}:
\begin{align*}
H(X,Y,\cdot\ )=0,\quad \forall X,Y\in \mathfrak{t}^k_E\in E.
\end{align*} 
The requirement that $H$ is admissible ensures that the T-dual bundle $\tilde{E}$ is in fact a torus fibration.  If $H$ is not admissible the T-dual is a non-commutative space \cite{Mat04}, and cannot be described by generalised geometry.  

The T-duality map $\tau_\cF:(\Omega^\bullet_{\sfT^k}(M),d_{H})\rightarrow (\Omega^\bullet_{\tilde{\sfT}^k}(\tilde{M}),d_{\tilde{H}})$, is given by the formula 
\begin{align*}
\tau_\cF(\varphi+i\cA\psi)=\int_{\sfT^k}e^\cF p^*(\varphi+i\cA\psi),
\end{align*}
where $(\varphi,\psi)$ is a mixed pair, $\cF=-\cA\tilde{\cA}$, and $\cA$, $\tilde{\cA}$, denote the connections specifying the tori $\sfT^k$ and $\tilde{\sfT}^k$ respectively.  The map $\tau_\cF$ can be seen as the combination of a pullback from $E$ to the correspondence space $E\times_M\tilde{E}$, a $B$-transformation $e^\cF$, and then the pushforward to $\tilde{E}$.  This can be viewed as a type of geometric Fourier transform.

The description of T-duality for generalised (almost) contact structures on the trivial bundle $E=M\times \RR$ is given in \cite{Ald13}.

Given the interpretation of generalised contact structures as $\sfS^1$-reduced generalised complex structures, $\sfT^k$-duality of generalised contact structures is $\sfT^{k+1}$-duality of the corresponding generalised complex structure.  In \cite{Cav10}, T-duality for circle bundle is considered as an example.  The killing vector generates a $\sfS^1$-foliation, and considering $\sfS^1$-invariant fields, the Courant bracket \eqref{Dorf} is reduced to \eqref{RedCourant}.  T-duality corresponds to the interchange $F,f\leftrightarrow H_2,g$.  Contact geometry corresponds to an extra $\sfS^1$-invariant reduction, but not the interchange and pushforward.

In \cite{Ald13} T-duality in the cone direction, $t$, is considered.  In this case the mixed pair $(\varphi,\psi)$ is mapped to the mixed pair $(\psi,\varphi)$.   A $b$-transformation is interpreted as a change in connection, and hence fibering for the $\sfS^1$-bundle defining the generalised contact structure.  An $a$-transformation corresponds to a choice of connection in the T-dual direction.  

\begin{proposition}
	T-duality maps a generalised coK\"ahler(-Einstein) structure to another generalised coK\"ahler(-Einstein) structure.
\end{proposition}
\begin{proof}
	T-duality preserves the pairing, and maps a mixed pair to another mixed pair. 
\end{proof}
\section{Contact line bundles vs reduction}\label{Linebundle}

It has recently been shown that generalised contact geometry has a conceptually nice description as generalised geometry on the generalised derivation bundle $\DD L\cong \fD L\oplus \fJ^1L$, for a (possibly non-trivial) line bundle $L$ \cite{Vit15a,Vit15b}.  This Section briefly outlines the description, and relates this to the current note.  In particular a generalised contact structure viewed as a reduced generalised complex structure, $\cJ|_{\sfS^1}$, is the $\sfS^1$-bundle version of the generalised complex structure $\cI\in\End(\DD L)$, and a mixed pair $(\varphi,\psi)$ are associated to a pure spinor $\varpi\in\Gamma(\wedge^\bullet\fJ^1 L,L)$.

Many interesting examples of contact structures are in fact non-coorientable, that is they are not defined by a globally defined contact one-form.  Instead contact structures are determined by a line bundle $L=TM/D$, as described in Section \ref{gencontactstruc}.  It is of interest to have a formalism that allows the description of non-trivial line bundles, which additionally makes the symmetries explicit.

The description of generalised contact bundles is given via the Atiyah (or gauge) algebroid, defined on $\DD L=\fD L\oplus \fJ^1 L$, where sections of $\fD L$ are derivations of $L$, and $\fJ^1L$ is the first jet bundle of $L$.  A derivation $\nabla\in\fD E$ has a unique {\emph symbol} $\sigma:\fD E\rightarrow TM$ such that, for $f\in C^\infty(M)$, $\lambda\in\Gamma(E)$, 
\begin{align*}
\nabla(f\lambda)=(\sigma\nabla)(f)\lambda+f\nabla\lambda=X(f)\lambda+f\nabla\lambda,
\end{align*}
where $X=\sigma(\nabla)$.  This makes it clear that $\fD E$ is part of the exact sequence
\begin{equation*}
\xymatrix{0\ar[r]& \mathfrak{gl}(E) \ar[r]  & \fD E \ar[r]^{\sigma}  & TM \ar[r]  & 0}. \label{DEsplit}
\end{equation*}
There is a natural Lie algebroid structure associated with $\fD E$, with the Lie bracket given by the commutator of derivations, and the anchor given by $\sigma$.  In the case of a line bundle the induced map on sections gives:
\begin{align}\label{splitDL}
\xymatrix{0\ar[r]& \Gamma(\mathfrak{gl}(L))\cong C^\infty(M) \ar[r]  & \Gamma(\fD L) \ar[r]^{\sigma}  & \Gamma(TM) \ar[r]  & 0},
\end{align}
giving $\Gamma(\fD L)\cong \Gamma(TM)\oplus C^\infty(M)$.  Sections $(X,f)\in(\Gamma(TM),C^\infty(M))$ are the line bundle version of the $\sfS^1$-invariant sections of $TP$ for $\sfS^1\hookrightarrow P\rightarrow M$.  

The 1-jet bundle $\fJ^1 E$ can be defined, at a point $p\in M$, by the equivalence relation in $\Gamma(E)$:
\begin{align*}e_1\sim e_2\leftrightarrow e_1(p)=e_2(p),\quad d\langle e_1,\zeta\rangle=d\langle e_2,\zeta\rangle,\quad \forall \zeta\in\Gamma(E^*).
\end{align*}
There exists $\pp:\fJ^1 E\rightarrow E$, such that $\ker(\pp)\cong\mathsf{Hom}(TM,E)$, giving
\begin{equation*}
\xymatrix{0\ar[r]& \mathsf{Hom}(TM,E) \ar[r]  & \fJ^1 E \ar[r]^{\pp}  & E  \ar[r] & 0}. \label{JEsplit}
\end{equation*}
In the case of a line bundle the induced map on sections gives:
\begin{align}\label{splitJL}
\xymatrix{0\ar[r]& \Gamma(\mathsf{Hom}(TM,L))\cong \Gamma(T^*M) \ar[r]  & \Gamma(\fJ^1 L) \ar[r]  & \Gamma(L)\cong C^\infty(M)  \ar[r] & 0},
\end{align}
giving $\Gamma(\fJ^1L)\cong \Gamma(T^*M)\oplus C^\infty(M)$.  Sections $(\xi,g)\in (\Gamma(T^*M), C^\infty(M))$ are the line bundle version of the $\sfS^1$-invariant sections of $T^*P$.  

It is shown in \cite{Che07} that $\fD E$ is $E$-dual to $\fJ^1 E$, there is a non-degenerate $E$-valued pairing $\langle\cdot,\cdot\rangle_E:\fD E\times \fJ^1 E \rightarrow E$.  For sections $\nabla\in\Gamma(\fD E)$ and $\chi=\sum fj^1e\in \Gamma(\fJ^1 E)$ the pairing is given by $\langle \nabla,\chi\rangle_E=\sum f\nabla(e)$.  The pairing between $\fD E$ and $\fJ^1 E$ has a geometric interpretation, $\langle \nabla,\chi \rangle$ can be viewed as the covariant derivation of $\chi$ with respect to $\nabla$. 

Given the $E$-valued pairing between $\fD E$ and $\fJ^1 E$, there is a natural $E$-Courant (more specifically an omni-Lie) algebroid defined on the generalised derivation bundle $\DD E$ given by:
\begin{equation*}
\xymatrix{0\ar[r]& \fJ^1 E \ar[r]  & \DD E \ar[r]^{\rho}  & \fD E  \ar[r] & 0}.
\end{equation*}
The definition and properties of omni-Lie and $E$-Courant algebroids can be found in \cite{Che07}, and \cite{Che08}, respectively.  The $E$-Courant algebroid can be viewed as a derived bracket for the differential $d_{\fD E}$ acting on the complex $\Omega^k_{E}:=\Gamma(\wedge^k\fJ^1 E,E)$:
\begin{align*}
d_{\fD E}\varpi(\nabla_0,\nabla_2,\dots,\nabla_k)=&\sum^k_{i=0}(-1)^i\nabla_i(\varpi(\nabla_0,\dots,\hat{\nabla}_i,\dots,\nabla_k))\\
&+\sum_{i<j}(-1)^{i+j}\varpi([\nabla_i,\nabla_j],\nabla_0,\dots,\hat{\nabla}_i,\dots,\hat{\nabla}_j,\dots,\nabla_k),\nonumber
\end{align*}
for $\nabla_i\in\Gamma(\fD E)$, $\varpi\in\Omega^k_E$,
where $\hat{\cdot}$ denotes omission.  The action of $\fL_{\nabla}\varpi:=d_{\fD E}\iota_{\nabla}\varpi+\iota_{\nabla}d_{\fD E}\varpi$, gives a Lie derivative on $\Gamma(\wedge^\bullet \fJ^1E,E)$, satisfying an analogue of the Cartan relations.

From this construction the omni-Lie algebroid on a line bundle $L\rightarrow M$ is given as: 
\begin{subequations}\label{omnilie}
	\begin{align}
	(\nabla_1,\psi_1)\circ^L(\nabla_2,\psi_2)=&\Big([\nabla_1,\nabla_2], \fL_{\nabla_1}\psi_2-\iota_{\nabla_2}d_{\fD L}\psi_1\Big);\label{omnibracket}\\
	\langle\langle (\nabla_1,\psi_1),(\nabla_2,\psi_2)\rangle\rangle=&\langle\nabla_1,\psi_2\rangle_L+\langle\nabla_2,\psi_1\rangle_L;\\
	\rho(\nabla,\psi)=&\nabla,
	\end{align}
\end{subequations}
for $\nabla\in\Gamma(\fD L)$, $\psi\in\Gamma(\fJ^1 L)$.  The bracket \eqref{omnibracket} can be identified with \eqref{RedCourant} with $H_3=0$.  In the case of a trivial line bundle $H_2=F=0$ has already been noted \cite{Vit15a}. If the line bundle is non-trivial then $F=H_2$ is given by the curvature of a connection specifying the bundle.  

Having identified the Courant algebroids \eqref{omnibracket} with \eqref{RedCourant}, the identification of generalised contact structures as generalised complex structures is straightforward.  The generalised complex structure $\cJ|_{\sfS^1}\in\End(TP\oplus T^*P)$ can be identified with $\cI\in \End(\DD L)$, satisfying $\cI^2=-\text{id}$, and $\cI^*=-\cI$, by splitting the sequences \eqref{splitDL} and \eqref{splitJL}. The generalised complex structure $\cI$ is identified with a Dirac structure $L_\cI\subset \DD L$, and described by a pure spinor $\varpi\in \Gamma(\wedge^\bullet\fJ^1L,L)$.  A choice of decomposition $\Gamma(\fJ^1L)=C^\infty(M)\oplus \Gamma(T^*M)$ coming from \eqref{splitJL}, induces a decomposition $\varpi\in\Gamma(\wedge^\bullet\fJ^1L,L)$ into a mixed pair $(\varphi,\psi)\in \Gamma(\wedge^\bullet T^*M,\wedge^{\bullet-1}T^*M)$.

\section*{Acknowledgements}
I would like to thank Peter Bouwknegt for many helpful discussions throughout the project, and for his comments on draft versions of this note.  This research was supported by the Australian government through an Australian Postgraduate Award, and the Australian Research Council’s Discovery Projects funding scheme (project DP150100008).




\begin{thebibliography}{99}
	\bibitem{Ald13}
	M.~Aldi, and D.~Grandini,
	{\it Generalized contact geometry and T-duality},
	\href{http://www.sciencedirect.com/science/article/pii/S0393044015000352?}{J. Geom. Phys. {\bf 92} (2015) 78--93},
	\href{https://arxiv.org/abs/1312.7471v3}{[{arXiv:1312.7471}]}.
	
	\bibitem{Bar11}
	D.~Baraglia,
	{\it Leibniz algebroids, twisting and exceptional generalized geometry},
	\href{http://www.sciencedirect.com/science/article/pii/S0393044012000174?via%3Dihub}{J. Geom. Phys {\bf 62} (2012) 903--934},
	\href{https://arxiv.org/abs/1101.0856}{[{arXiv:1101.0856}]}.
	
	\bibitem{Baz14}
	G.~Bazzoni, and J.~Oprea,
	{\it On the structure of co-K\"ahler manifolds},
	\href{https://link.springer.com/article/10.1007/s10711-013-9869-7}{ J. Geom Dedicata {\bf 170} (2014)},
	\href{https://arxiv.org/abs/1209.3373}{[{arXiv:1209.3373}]}.
	
	\bibitem{Bou05}
	P.~Bouwknegt,
	{\it unpublished}.
	
	\bibitem{Bou10}
	P.~Bouwknegt,
	{\it Lectures on Cohomology, T-duality and Generalized Geometry}, 
	in the proceedings of the the Summer School ``New Paths Towards 
	Quantum Gravity'', Holb\ae k, Denmark, 10-16 May 2008, 
	Editors: B. Booss-Bavnbek, G. Esposito and M. Lesch, 
	Lecture Notes in Physics {\bf 807} (2010) 261--311.
	
	\bibitem{Bou04a}
	P.~Bouwknegt, J.~Evslin, and V.~Mathai,
	{\it T-duality: topology change from $H$-flux},
	\href{http://link.springer.com/article/10.1007/s00220-004-1115-6}{Comm. Math. Phys.  {\bf 249} (2004) 383--415},
	\href{https://arxiv.org/abs/hep-th/0306062}{[{arXiv:hep-th/0306062}]}.
	
	\bibitem{Bou04b}
	P.~Bouwknegt, J.~Evslin, and V.~Mathai,
	{\it Topology and $H$-flux of T-dual manifolds},
	\href{http://link.aps.org/doi/10.1103/PhysRevLett.92.181601}{Phys. Rev. Lett. {\bf 98} (2004) 181601},
	\href{https://arxiv.org/abs/hep-th/0312052}{[{arXiv:hep-th/0312052}]}.
	
	
	\bibitem{Cav04}
	G.~Cavalcanti, and M.~Gualtieri,
	{\it Generalized complex manifolds on nilmanifolds},
	\href{http://intlpress.com/site/pub/pages/journals/items/jsg/content/vols/0002/0003/a005/index.html}{J. Symplectic Geom. {\bf 3} (2004) 393--410},
	\href{https://arxiv.org/abs/math/0404451}{[{arXiv:math/0404451}]}.
	
	\bibitem{Cav10}
	G.~Cavalcanti, and M.~Gualtieri,
	{\it Generalized complex geometry and T-duality},
	A Celebration of the Mathematical Legacy of Raoul Bott (CRM Proceedings \& Lecture Notes), American Mathematical Society (2010) 341--366,
	\href{https://arxiv.org/abs/1106.1747}{[{arXiv:1106.1747}]}.
	
	\bibitem{Che07}
	Z.~Chen, and Z.-J. Liu,
	{it Omni-Lie algebroids},
	\href{http://www.sciencedirect.com/science/article/pii/S0393044010000306}{
	J. Geom. Phys. {\bf 60} (2010) 799--808},
	\href{https://arxiv.org/abs/0710.1923}{[{arXiv:0710.1923}]}.
	
	\bibitem{Che08}
	Z.~Chen, Z.~Liu, and Y.~Sheng,
	{\it E-Courant algebroids},
	\href{https://academic.oup.com/imrn/article-lookup/doi/10.1093/imrn/rnq053}{Int. Math. Res. Not. {\bf 22} (2010) 4334--4376},
	\href{https://arxiv.org/abs/0805.4093}{[{arXiv:0805.4093}]}.
	
%
	\bibitem{Gat84}
	S.J.~Gates, C.M.~Hull, and M.~Ro\v{c}ek,
	{\it Twisted multiplets and new supersymmetric non-linear $\sigma$-models},
	\href{http://www.sciencedirect.com/science/article/pii/0550321384905923}{Nuclear Physics B {\bf 248}(1) (1984) 157--186}.
	
	\bibitem{Gom13}
	R.~Gomez, and J.~Talvacchia,
	{\it On products of generalized geometries},
	\href{https://link.springer.com/article/10.1007/s10711-014-0036-6}{Geometriae Dedicata {\bf 175} (2015) 211--218},
	\href{https://arxiv.org/abs/1311.0381}{[{arXiv:1311.0381}]}.
	
	\bibitem{Gom15}
	R.~Gomez, and J.~Talvacchia,
	{\it Generalized coK\"ahler Geometry and Application to Generalized K\"ahler structures},
	\href{https://www.sciencedirect.com/science/article/pii/S039304401500217X?via%3Dihub}{J. Geom. Phys. {\bf 98} (2015) 493--503},
	\href{https://arxiv.org/abs/1502.07046}{[arXiv:1502.07046]}.
	
	\bibitem{Gra13}
	J.~Grabowski,
	{\it Graded contact bundles and contact Courant algebroids},
	\href{http://www.sciencedirect.com/science/article/pii/S039304401300017X}{J. Geom. Phys. {\bf 68} (2013) 27--58},
	\href{https://arxiv.org/abs/1112.0759}{[{arXiv:1112.0759~[math.DG]}]}.
	
	\bibitem{Gua04}
	M.~Gualtieri,
	{\it Generalized complex geometry},
	\href{http://annals.math.princeton.edu/2011/174-1/p03}{Ann. Math. {\bf 174} (2011) 75--123},
	\href{http://lanl.arxiv.org/abs/math/0401221v1}{[{arXiv:math/0401221~[math.DG]}]}.
	
	\bibitem{Hit02}
	N.~Hitchin,
	{\it Generalized Calabi-Yau manifolds},
	\href{https://academic.oup.com/qjmath/article-lookup/doi/10.1093/qmath/hag025}{Quart. J. Math {\bf 54} (2003) 281--308}, 
	\href{https://arxiv.org/abs/math/0209099}{[{arXiv:math/0209099}]}.
	
	\bibitem{Igl01}
	D.~Iglesias, and J.C.~Marrero,
	{\it Lie algebroid foliations and $\cE^1(M)$-Dirac structures},
	\href{http://iopscience.iop.org/article/10.1088/0305-4470/35/18/307/meta;jsessionid=9EA9A84187D41DA0C5A7FCFB9486DF38.c4.iopscience.cld.iop.org}{J. Phys. A:Math.Gen. {\bf 35} (2002)},
	\href{https://arxiv.org/abs/math/0106086}{[{arXiv:math/0106086~[math.DG]}]}.
	
	\bibitem{Igl04}
	D.~Iglesias, and A.~Wade,
	{\it Contact manifolds and generalized complex structures},
	\href{http://www.sciencedirect.com/science/article/pii/S0393044004001007?}{J. Geom. Phys. {\bf 53} (2005) 249--258},
	\href{https://arxiv.org/abs/math/0404519}{[{arXiv:math/0404519~[math.DG]}]}.
	
	\bibitem{Kos04}
	Y.~Kosmann-Schwarzbach,
	{\it Derived Brackets},
	\href{https://link.springer.com/article/10.1007%2Fs11005-004-0608-8}{Lett. Math. Phys. {\bf 69} (2004) 61--87},
	\href{https://arxiv.org/abs/math/0312524}{[{arXiv:math/0312524}]}.	
	
	\bibitem{Li08}
	H.~Li,
	{\it Topology of Co-symplectic/Co-K\"ahler manifolds},
	\href{http://intlpress.com/site/pub/pages/journals/items/ajm/content/vols/0012/0004/a007/index.html}{Asian J. Math. {\bf 12(4)} (2008) 527-–544}.
	
	\bibitem{Mat04}
	V.~Mathai, and J.~Rosenberg,
	{\it T-Duality for Torus Bundles with $H$-fluxes via Noncommutative Topology},
	\href{https://link.springer.com/article/10.1007/s00220-004-1159-7}{Comm. Math. Phys. {\bf 253(3)} (2005) 705--721},
	\href{https://arxiv.org/abs/hep-th/0401168}{[{arXiv:hep-th/0401168}]}.
	
%
	\bibitem{Poo09}
	Y.S.~Poon, and A.~Wade,
	{\it Generalized Contact Structures},
	\href{http://onlinelibrary.wiley.com/doi/10.1112/jlms/jdq069/abstract;jsessionid=FD934CE172DE9BDF05DB3D39C70A50F9.f04t03}{J. London Math. Soc. (2011) 333--352},
	\href{https://arxiv.org/abs/0912.5314v1}{[arXiv:0912.5314~[math.DG]]}.
	
	\bibitem{Sek12}
	K.~Sekiya,
	{\it Generalized almost contact structures and generalized Sasakian structures},
	\href{http://projecteuclid.org/euclid.ojm/1427202871}{Osaka J. Math. {\bf 52} (2015)},
	\href{https://arxiv.org/abs/1212.6064v1}{[{arXiv:1212.6064~[math.DG]}]}.
	
	\bibitem{Sev98}
	P.~\v{S}evera,
	{\it Letters to Alan Weinstein about Courant algebroids},
	\href{https://arxiv.org/abs/1707.00265}{[{arXiv:1707.00265}]}.
	
	
	\bibitem{Vai07}
	I.~Vaisman,
	{\it Dirac structures and generalized complex structures on $TM\times \RR^h$},
	\href{https://www.degruyter.com/view/j/advg.2007.7.issue-3/advgeom.2007.029/advgeom.2007.029.xml}{Adv. Geom. {\bf 7} (2007) 453--474},
	\href{https://arxiv.org/abs/math/0607216}{[{arXiv:math/0607216}]}.
	
	\bibitem{Vai08}
	I.~Vaisman, 
	{\it Generalized CRF-structures},
	\href{https://link.springer.com/article/10.1007/s10711-008-9239-z}{Geom. Dedicata {\bf 133} (2008) 129--154},
	\href{https://arxiv.org/abs/0705.3934}{[{arXiv:0705.3934}]}.
	
	
	\bibitem{Vit15a}
	L.~Vitagliano,
	{\it Dirac-Jacobi bundles},
	\href{https://arxiv.org/abs/1502.05420}{[{arXiv:1502.05420}]}.
	
	\bibitem{Vit15b}
	L.~Vitagliano, and A.~Wade,
	{\it Generalized contact bundles},
	\href{http://www.sciencedirect.com/science/article/pii/S1631073X15003374?via%3Dihub}{Comptes Rendus Mathematique {\bf 354} (2016) 313--317},
	\href{https://arxiv.org/abs/1507.03973}{[{arXiv:1507.03973}]}.
	

\end{thebibliography}
\end{document}